\documentclass[review]{article}
\usepackage[utf8]{inputenc}
\usepackage{fullpage}
\usepackage{amsmath}
\usepackage{amssymb}
\usepackage{amsthm}
\usepackage[colorlinks]{hyperref}
\usepackage{lineno}
\usepackage{tikz}
\usepackage{enumerate}

\usepackage{xcolor}

\newcommand{\Poly}{\textsc{P}}
\newcommand{\NP}{\textsc{NP}}
\newcommand{\W}{\textsc{W}}

\newcommand{\hn}{\textrm{hn}}

\newcommand{\ohn}{\overrightarrow{\textrm{hn}}}
\newcommand{\ogn}{\overrightarrow{\textrm{gn}}}
\DeclareMathOperator{\phn}{hn^+}
\DeclareMathOperator{\pgn}{gn^+}
\DeclareMathOperator{\mhn}{hn^-}
\DeclareMathOperator{\mgn}{gn^-}
\newcommand{\oracf}{\overrightarrow{C4}}
\DeclareMathOperator{\ext}{Ext}

\newtheorem{thm}{Theorem}[section]
\newtheorem{prop}[thm]{Proposition}
\newtheorem{lem}[thm]{Lemma}
\newtheorem{cor}[thm]{Corollary}
\newtheorem{conj}[thm]{Conjecture}

\modulolinenumbers[5]

\begin{document}
    

\title{Hull and geodetic numbers for some classes of oriented graphs\thanks{This work was partially supported by a CNPq/Funcap project PNE-0112-00061.01.00/16 SPU N: 4543945/2016, a CNPq Universal project 401519/2016-3, CNPq grants 310234/2015-8 and 130467/2018-9, CAPES STIC-AmSud 88881.197438/2018-01 and CAPES-PrInt 88887.466468/2019-00.}}


\author{J.~Ara\'ujo and P.S.M.~Arraes}


\maketitle

\begin{abstract}
    Let \(D\) be an orientation of a simple graph.
    Given $u,v\in V(D)$, a directed shortest \((u,v)\)-path is a \( (u,v) \)-geodesic.
    \( S \subseteq V(D) \) is convex if, for every \( u,v \in S \), the vertices in each \( (u,v) \)-geodesic and in each \( (v,u) \)-geodesic are in \( S \).
    For each \( S \subseteq V(D) \) the (convex) hull of \( S \), denoted by \([S]\), is the smallest convex set containing \( S \).
    \( S \subseteq V(D) \) is a hull set if \([S] = V(D) \).
    \( S \subseteq V(D) \) is a geodetic set of \( D \) if each vertex of \( D \) lies in a \( (u,v) \)-geodesic, for some \( u,v \in S \).
    The cardinality of a minimum hull set (resp. geodetic set) of $G$ is the hull number (resp. geodetic number) of \( D \), denoted by \( \ohn (D) \) (resp. $\ogn(D)$).
    
    We first show a tight upper bound on $\ohn(D)$. Given $k\in\mathbb{Z}_+^*$, we prove that deciding if \(\ohn(D)\leq k\) is \NP-complete when $D$ is an oriented partial cube; and if \(\ogn(D)\leq k\) is \W[2]-hard parameterized by \(k\) and has no $(c \cdot \ln n)$-approximation algorithm, unless \Poly =\NP, even if $D$ has an underlying graph that is bipartite or split or cobipartite.
    We also show polynomial-time algorithms to compute $\ohn(D)$ and $\ogn(D)$ when $D$ is an oriented cactus.
\end{abstract}


\section{Introduction}
\label{sec.intro}

For basic notions on graph theory and computational complexity, the reader is referred to~\cite{BM2008,GJ90, BG2008}.
All graphs in this work are simple and finite, unless explicitly stated otherwise.

Although the first papers related to convexity in graphs study directed graphs~\cite{CIT,EITIST,SROST}, most of the papers we can find in the literature about graph convexities deal with undirected graphs. 
For instance, the hull and geodetic numbers with respect to undirected graphs~\cite{THNOAG,TGNOAG} were first studied in the literature around a decade before their corresponding directed versions~\cite{THNOAOG,TGNOAOG}.

An oriented graph \(D\) is an orientation of a simple graph.
A directed path from \(u\) to \(v\) with minimum number of arcs in \(D\) is a \( (u,v) \)-geodesic for every \(u,v\in V(D)\).
Let \( \mathcal{P} (S) = 2^S \) denote the family of all subsets of the set \( S \).
The interval function \( I: \mathcal{P} (V(D)) \to \mathcal{P} (V(D)) \) applied on a set \( S \subseteq V(D) \) with at least two elements, returns the set \( I(S) \) consisting of the vertices of \( S \) and every vertex lying in some \( (u,v) \)-geodetic such that \( u,v \in S \).
If \( |S| \in \{ 0,1 \} \), we have \( I(S) = S \).

A set \( S \subseteq V(D) \) is (geodesically) convex if \( I(S) = S \); in other words, for every \( u,v \in S \) all the vertices in each \( (u,v) \)-geodesic and in each \( (v,u) \)-geodesic are in \( S \).
For every \( S \subseteq V(D) \) the (convex) hull of \( S \) is the smallest convex set containing \( S \) and it is denoted by \( [S] \).
A hull set of \( D \) is a set \( S \subseteq V(D) \) whose hull is \( V(D) \).
The cardinality of a minimum hull set is the hull number of \( D \) and it is denoted by \( \ohn (D) \).
A geodetic set of \( D \) is a set \( S \subseteq V(D) \) such that \( I(S) = V(D) \).
Analogously, the cardinality of a minimum geodetic set is the geodetic number of \( D \) and it is denoted by \( \ogn (D) \).

Because an oriented graph \( D \) has no symmetric arcs, the parameters $\ohn(D)$ and $\ogn(D)$ are not equivalent to their undirected versions.
For instance, the hull and geodetic numbers of a path $P$ on $2k$ edges, for some positive integer $k$, are both equal to two in the undirected version, while if $D$ is an orientation of $P$ we can have both \( \ohn(D) \) and \( \ogn(D) \) ranging from \( 2 \) to \( 2k+1 \).

With respect to the directed case, most results in the literature provide bounds on the maximum and minimum values of $\ohn(D(G))$ and $\ogn(D(G))$ among all possible orientations $D(G)$ of a given undirected simple graph $G$~\cite{THNOAOG,TGNOAOG,OCGAHNIG}.

It is important to emphasize the results on the parameter $\phn(G)$, the upper orientable hull number of a graph \( G \), since these are the only ones related to the upper bounds we present.
Such parameter is defined in~\cite{THNOAOG} as the maximum value of $\ohn(D(G))$ among all possible orientations $D(G)$ of a simple graph $G$.
In the same article, the authors provide a necessary and sufficient condition for \( \phn(G) = n(G) \), with \( G \) being a non-oriented graph.
They also compare this parameter with others, such as the lower orientable hull number (\( \mhn(G) \)) and the lower and upper orientable geodetic numbers (\( \mgn(G) \) and \( \pgn(G) \) respectively), defined analogously.

There are also few results about some related parameters: the forcing hull and geodetic numbers~\cite{TFHAFGNOG,TFGNOAG}, the pre-hull number~\cite{OTGPHNOAG} and the Steiner number~\cite{OTSGAHNOG} are a few examples.

In this work, we first present a general tight upper bound on the hull number of an arbitrary oriented graph, in Section~\ref{section.st}.
Note that such bound is also a tight upper bound for $\phn(G)$.

Then, we consider as input an oriented graph $D$ and we study the computational complexity of determining $\ohn(D)$ and $\ogn(D)$, when the underlying graph of $D$ belongs to some particular graph class.
Up to our best knowledge, this is the first work to consider such questions. 

It is known that determining the hull number of an undirected partial cube is \NP-hard~\cite{CIPCTHN}.
In Section~\ref{section.bip}, we show that such result can be used to prove that determining whether $\ohn(D)\leq k$, when $D$ is an oriented partial cube, is \NP-complete.
Although the proof requires a careful analysis, the idea is quite simple: by replacing each edge of a partial cube $G$ with a directed $C_4$, we obtain an oriented graph $D$ whose underlying graph is a partial cube, and whose hull number $\ohn(D)$ is the same as $\hn(G)$.
It is important to recall that partial cubes are bipartite graphs.

In Section~\ref{sec.setcover}, we reduce the well-known Set Cover problem to that of determining whether $\ogn (D) \leq k$, even if $D$ is a directed acyclic graph (DAG) whose underlying graph is either bipartite or cobipartite; or if the underlying graph is split (the construction has directed cycles).
Such reduction implies that this problem is not only \NP-complete, but it is \W[2]-hard (parameterized by $k$) and does not admit any approximation algorithm with ratio $c \cdot \ln n$, unless $\Poly = \NP$, as these hardness results hold for the Set Cover problem~\cite{Karp1972,RS1997,DF2012}.

Finally, we prove in Section~\ref{section.tc} that $\ohn(D)$ and $\ogn(D)$ can be computed in polynomial time if $D$ is a cactus, i.e. a graph whose blocks are either edges or induced cycles.

In Section~\ref{section.conc}, we present avenues for further research.

\section{Preliminaries}
\label{sec.pre}

For a positive integer \( k \) we denote \( [k] := \{ 1, \ldots ,k \} \), i.e. the set of naturals smaller than or equal to \( k \).
Given a set \( S \) we denote by \( \mathcal{P} (S) = 2^S\) the family of all subsets of \( S \).
Let \( G \) be a graph or digraph, \( n(G) \) denotes the number of vertices of \( G \).

An \emph{oriented graph} $D$ is an orientation of a simple graph. With a slight abuse of notation, we also use the term subgraph between two oriented graphs, meaning subdigraph.
Given an oriented graph \( D \), a (directed)  \( (u,v) \)-\textit{path} \( P \) is a subgraph of \( D \) such that \( V(P) = \{ u=u_0, u_1, \ldots ,u_k = v \} \) and \( A(P) = \{ (u_{i-1},u_i) \mid i \in \{1,\ldots, k\} \} \).
We can also denote it by \( (u,u_1, \ldots ,u_{k-1},v) \); to represent a path in a non-oriented graph \( G \) we remove the parentheses. 
When \( w \) is a vertex of \( P \) different than \( u \) and \( v \) we say that it is an \textit{internal} vertex of \( P \).
The set of internal vertices of $P$ we call the \emph{interior} of $P$.
The length of a path \( P= (u,u_1, \ldots ,u_{k-1},v) \) is \(k\).
An \( (u,v) \)-path that uses the least number of arcs possible is called an \( (u,v) \)-\textit{geodesic}. 
We denote its length by \( d_D(u,v) \) which represents the \emph{distance} in \( D \) from \( u \) to \( v \).
Notice that \( d_D(u,v) \) might not be equal to \( d_D(v,u) \), since we are dealing with directed graphs.
In the sequel, whenever $D$ is clear in the context we only use  \( d(u,v) \).

Given a vertex \( v \in V(D) \) we define \( N^-(v) := \{ u \in V(D) \mid (u,v) \in A(D) \} \) and \( N^+(v) := \{ u \in V(D) \mid (v,u) \in A(D) \} \).
Moreover we respectively define the \textit{indegree} and the \textit{outdegree} of \( v \) by \( d^-(v) := |N^-(v)| \) and \( d^+(v) := |N^+(v)| \).

For two oriented graphs \( D_1,D_2 \) such that \( D_1 \) is a subgraph of \( D_2 \) we denote this fact by \( D_1 \subseteq D_2 \).
Given an oriented graph \( D \) and \( C \subseteq D \) such that its underlying graph is a cycle, we say that \( C \) is simply a \textit{cycle}; the fact that it is oriented is already implied by being a subgraph of \( D \).
However, when \( C \) is such that \( V(C) = \{ v_1, \ldots ,v_n \} \) and \( A(D) = \{ (v_1,v_2), \ldots ,(v_{n-1},v_n),(v_n,v_1) \} \) we say that it is a \textit{directed} cycle.

The \textit{interval function} \( I \colon \mathcal{P}(V(D)) \to \mathcal{P}(V(D)) \) satisfies that, for each vertex set \( S \subseteq V(D) \) with at least two elements, \( I(S) \) is the set of all vertices in an \( (u,v) \)-geodesic (\( u \) and \( v \) included), for every \( u,v \in S \); when \( S \) has less than two elements, we have \( I(S) = S \).
For every positive integer \( n \), we recursively define $I^0(S) = S$ and \( I^n(S) := I(I^{n-1}(S)) \).
A subset \( S\subseteq V(D) \) is \textit{convex} when \( I(S) = S \); if this happens, we say that \( S^{\complement} = V(D) \setminus S \) is \textit{co-convex}.
The \textit{convex hull} of \( S \) is the smallest convex set which contains \( S \) and is denoted by \( [S] \).
There are two interesting properties for this set.
One is that it is the intersection of all convex sets containing \( S \).
A noteworthy consequence of this fact is that if \( S \) does not intersect a given co-convex set, then its convex hull also does not intersect it.
The other is that it is obtained when we iterate the interval function on \( S \) until we reach a convex set, \( I^k(S) = [S] \).
Assuming that \( V(D) \) is finite, the convex hull for every subset of \( V(D) \) is well-defined.


If the convex hull of \( S \) is \( V(D) \), we say that \( S \) is a \textit{hull set} of \( D \).
When \( S \) is a hull set of minimum cardinality, the \textit{hull number} of \( D \) is defined as \( \ohn (D) = |S| \)~\cite{THNOAOG}.
Similarly, if \( I(S)=V(D) \) we say that \( S \) is a \textit{geodetic set} of \( D \).
When \( S \) is a geodetic set of minimum cardinality, the \textit{geodetic number} of \( D \) is defined as \( \ogn (D) = |S| \)~\cite{TGNOAOG}.
Notice that a geodetic set is also a hull set, therefore every assertion we make for all hull sets of a given oriented graph \( D \), in particular, is also valid for all geodetic sets.

Now that we have presented the main parameters of our research, we define a very important type of vertex.
It was first introduced in~\cite{THNOAG} for the undirected case, however we use the definitions given in~\cite{THNOAOG}.
A vertex \( v \in V(D) \) is called \textit{extreme} if it is of one of the three types below:
\begin{enumerate}[1.]
    \item Transmitter (source): \( d^-(v) = 0 \) and \( d^+(v) \geq 0 \);
    \item Receiver (sink): \( d^-(v) \geq 0 \) and \(d^+(v) = 0 \);
    \item Transitive: \( d^-(v) > 0 \), \( d^+(v) > 0 \) and \( (u,w) \in A(D) \) for every \( u \in N^-(v) \) and \( w \in N^+(v) \).
\end{enumerate}
We denote the set of extreme vertices of an oriented graph \( D \) as \( \ext (D) \).
For undirected graphs we have the \textit{simplicial} vertices, which are the ones with a clique for neighborhood; analogously, the set of simplicial vertices of a graph \( G \) is denoted by \( \ext (G) \).
What is so interesting about the extreme vertices is that they must be in every hull set of the oriented graph, as shown in~\cite{THNOAOG}.
There is a similar result for the simplicials in~\cite{OTCOTHNOAG}.

\section{Upper bound for the hull number}
\label{section.st}

Given an oriented graph \( D \), we already know that the extreme vertices must be in every hull set.
Sometimes they are sufficient, as we show in Section \ref{section.tc} for oriented trees, but that is not always the case.
Thus, we attempted to iteratively add non-extreme vertices to an initial set \( S_0 := \ext (D) \), until we reached a positive integer \( k \) such that \( S_k \) is a hull set of \( D \).

\begin{thm}\label{thm.bound23}
    For every oriented graph \( D \), we have: \[ \ohn (D) \leq |\ext (D)| + \frac{2}{3}|V(D) \setminus \ext (D)|.\]
\end{thm}

\begin{proof}
    Let \( S_0 = \ext (D) \).
    We iteratively define \( S_i \) for \( i>0 \) such that \( S_{i-1} \subsetneq S_i \), \( [S_{i-1}] \subsetneq [S_i] \) and
    \begin{equation}\label{eq.bound23}
        |S_i| \leq |\ext (D)| + \frac{2}{3}|[S_i] \setminus \ext (D)|.
    \end{equation}
    Since we increase the cardinality of \( [S_i] \) at each step and \( D \) is finite, there is a positive integer \( k \) such that \( S_k \) is a geodetic set of \( D \) for which (\ref{eq.bound23}) holds.
    The upper bound for \( \ohn (D) \) is a direct consequence of (\ref{eq.bound23}) when \( i=k \) and the fact that \( \ohn (D) \leq |S_k| \).

    If $[S_0] = [\ext(D)] = V(D)$, then we have nothing to prove. So, assume that \( [S_{i-1}] \neq V(D) \) and choose \( v \in V(D) \setminus [S_{i-1}] \) arbitrarily.
    Since \( v \) is not extreme (because \( \ext (D) = S_0 \subset S_{i-1} \)), there are \( v_1,v_2 \in V(D) \) such that \( (v_1,v),(v,v_2) \in A(D) \) and \( (v_1,v_2) \notin A(D) \).
    This means that \( (v_1,v,v_2) \) is a geodesic in \( D \) and therefore \( v \in I( \{ v_1,v_2 \} ) \).
    Notice that \( v_1,v_2 \) cannot be both in \( [S_{i-1}] \), else we would have \( v \in [S_{i-1}] \).
    Thus at most one of them lies in \( [S_{i-1}] \).
    In case exactly one of them belongs to \([S_{i-1}]\), say \( v_1 \in [S_{i-1}] \) and \( v_2 \notin [S_{i-1}] \), we define \( S_i := S_{i-1} \cup \{ v_2 \} \), which implies \( [S_{i-1}] \cup \{ v_1,v \} \subset [S_i] \) and
    \begin{align*}
        |S_i| - |\ext (D)| &= |S_{i-1}| + 1 - |\ext (D)| \leq \frac{2}{3}|[S_{i-1}] \setminus \ext (D)| + 1\\
        &< \frac{2}{3}\left( |[S_{i-1}] \setminus \ext (D)| + 2 \right) \leq \frac{2}{3} |[S_{i-1}] \setminus \ext (D)|.
    \end{align*}
    If both \( v_1,v_2 \notin [S_{i-1}] \), we define \( S_i := S_{i-1} \cup \{ v_1,v_2 \} \), which implies \( [S_{i-1}] \cup \{ v_1,v_2,v \} \subset [S_i] \) and
    \begin{align*}
        |S_i| - |\ext (D)| &= |S_{i-1}| - |\ext (D)| + 2 \leq \frac{2}{3}|[S_{i-1}] \setminus \ext (D)| + 2\\
        &= \frac{2}{3}\left( |[S_{i-1}] \setminus \ext (D)| + 3 \right) \leq \frac{2}{3} |[S_{i-1}] \setminus \ext (D)|.
    \end{align*}
\end{proof}

In order to present our tight example, we need to introduce a product of oriented graphs.
The \textit{lexicographic product} of $D$ by $D'$, denoted by $D \circ D'$, is the oriented graph which satisfies $V(D \circ D') = V(D) \times V(D')$ and $((u_1,v_1),(u_2,v_2)) \in A(D \circ D')$ if, and only if, either $(u_1,u_2) \in A(D)$ or $u_1 = u_2$ and $(v_1,v_2) \in A(D')$.
In other words, for each vertex $v \in V(D)$ we take a copy of $D'$, namely $D'_v$, and if $(u,v) \in A(D)$, then we add the arcs $(u',v')$, for each $u' \in V(D'_u)$ and each $v' \in V(D'_v)$.

\begin{prop}
    \label{prop:tightexample_bounds}
    For any positive integer $k$, there exists a tournament $D$ without extreme vertices on $n=3k$ vertices such that $\ohn(D)=\frac{2}{3} n$.
\end{prop}
\begin{proof}
Let \( \overrightarrow{K_k} \) be a transitive tournament, i.e. an oriented \( K_k \) such that \( V(\overrightarrow{K_k}) = \{ 1, \ldots , k \} \) and \( A(\overrightarrow{K_k}) = \{ (i,j) \mid i,j \in V(\overrightarrow{K_k}) \textrm{ and } i<j \} \), and let \( \overrightarrow{C_3} \) be a directed \( C_3 \) with vertex set \( \{ u,v,w \} \).
Let \( D = \overrightarrow{K_k} \circ \overrightarrow{C_3} \) (see Figure~\ref{fig:proposition.tour} for an illustration of \( D \) when \( k=5 \)).

\begin{figure}[!htb]
    \begin{center}
        \begin{tikzpicture}
            \node[circle,fill=black,inner sep=0pt,minimum size=0.2cm, name=a3] at (-3,3.08) {};
            \node[circle,fill=black,inner sep=0pt,minimum size=0.2cm, name=a2] at (-4.62,1.90) {};
            \node[circle,fill=black,inner sep=0pt,minimum size=0.2cm, name=a4] at (-1.38,1.90) {};
            \node[circle,fill=black,inner sep=0pt,minimum size=0.2cm, name=a1] at (-4.02,0) {};
            \node[circle,fill=black,inner sep=0pt,minimum size=0.2cm, name=a5] at (-1.98,0) {};
            \node[circle,fill=black,inner sep=0pt,minimum size=0.2cm, name=d2] at (3,3.55) {};
            \node[circle,fill=black,inner sep=0pt,minimum size=0.2cm, name=b2] at (2.6,2.85) {};
            \node[circle,fill=black,inner sep=0pt,minimum size=0.2cm, name=e2] at (3.4,2.85) {};
            \node[circle,fill=black,inner sep=0pt,minimum size=0.2cm, name=d3] at (4.62,2.37) {};
            \node[circle,fill=black,inner sep=0pt,minimum size=0.2cm, name=b3] at (4.22,1.67) {};
            \node[circle,fill=black,inner sep=0pt,minimum size=0.2cm, name=e3] at (5.02,1.67) {};
            \node[circle,fill=black,inner sep=0pt,minimum size=0.2cm, name=d4] at (1.38,2.37) {};
            \node[circle,fill=black,inner sep=0pt,minimum size=0.2cm, name=b4] at (0.98,1.67) {};
            \node[circle,fill=black,inner sep=0pt,minimum size=0.2cm, name=e4] at (1.78,1.67) {};
            \node[circle,fill=black,inner sep=0pt,minimum size=0.2cm, name=d5] at (4.02,0.47) {};
            \node[circle,fill=black,inner sep=0pt,minimum size=0.2cm, name=b5] at (3.62,-0.23) {};
            \node[circle,fill=black,inner sep=0pt,minimum size=0.2cm, name=e5] at (4.42,-0.23) {};
            \node[circle,fill=black,inner sep=0pt,minimum size=0.2cm, name=d1] at (1.98,0.47) {};
            \node[circle,fill=black,inner sep=0pt,minimum size=0.2cm, name=b1] at (1.58,-0.23) {};
            \node[circle,fill=black,inner sep=0pt,minimum size=0.2cm, name=e1] at (2.38,-0.23) {};
            \draw (3,3.08) node [thick, minimum size=1.5cm, draw, circle, name=c3] {};
            \draw (1.38,1.90) node [thick, minimum size=1.5cm, draw, circle, name=c2] {};
            \draw (4.62,1.90) node [thick, minimum size=1.5cm, draw, circle, name=c4] {};
            \draw (1.98,0) node [thick, minimum size=1.5cm, draw, circle, name=c1] {};
            \draw (4.02,0) node [thick, minimum size=1.5cm, draw, circle, name=c5] {};
            \draw[black, thick, ->, >=stealth] (a1) -- (a2);
            \draw[black, thick, ->, >=stealth] (a1) -- (a3);
            \draw[black, thick, ->, >=stealth] (a1) -- (a4);
            \draw[black, thick, ->, >=stealth] (a1) -- (a5);
            \draw[black, thick, ->, >=stealth] (a2) -- (a3);
            \draw[black, thick, ->, >=stealth] (a2) -- (a4);
            \draw[black, thick, ->, >=stealth] (a2) -- (a5);
            \draw[black, thick, ->, >=stealth] (a3) -- (a4);
            \draw[black, thick, ->, >=stealth] (a3) -- (a5);
            \draw[black, thick, ->, >=stealth] (a4) -- (a5);
            \draw[black, thick, ->, >=stealth] (c1) -- (c2);
            \draw[black, thick, ->, >=stealth] (c1) -- (c3);
            \draw[black, thick, ->, >=stealth] (c1) -- (c4);
            \draw[black, thick, ->, >=stealth] (c1) -- (c5);
            \draw[black, thick, ->, >=stealth] (c2) -- (c3);
            \draw[black, thick, ->, >=stealth] (c2) -- (c4);
            \draw[black, thick, ->, >=stealth] (c2) -- (c5);
            \draw[black, thick, ->, >=stealth] (c3) -- (c4);
            \draw[black, thick, ->, >=stealth] (c3) -- (c5);
            \draw[black, thick, ->, >=stealth] (c4) -- (c5);
            \draw[black, thick, ->, >=stealth] (b1) -- (d1);
            \draw[black, thick, ->, >=stealth] (d1) -- (e1);
            \draw[black, thick, ->, >=stealth] (e1) -- (b1);
            \draw[black, thick, ->, >=stealth] (b2) -- (d2);
            \draw[black, thick, ->, >=stealth] (d2) -- (e2);
            \draw[black, thick, ->, >=stealth] (e2) -- (b2);
            \draw[black, thick, ->, >=stealth] (b3) -- (d3);
            \draw[black, thick, ->, >=stealth] (d3) -- (e3);
            \draw[black, thick, ->, >=stealth] (e3) -- (b3);
            \draw[black, thick, ->, >=stealth] (b4) -- (d4);
            \draw[black, thick, ->, >=stealth] (d4) -- (e4);
            \draw[black, thick, ->, >=stealth] (e4) -- (b4);
            \draw[black, thick, ->, >=stealth] (b5) -- (d5);
            \draw[black, thick, ->, >=stealth] (d5) -- (e5);
            \draw[black, thick, ->, >=stealth] (e5) -- (b5);
            \draw[black, thick, ->] (-0.7,1.54) -- (0.2,1.54);
        \end{tikzpicture}
    \end{center}
    \caption{Tight example to Theorem~\ref{thm.bound23} assuming \( k=5 \).}
    \label{fig:proposition.tour}
\end{figure}
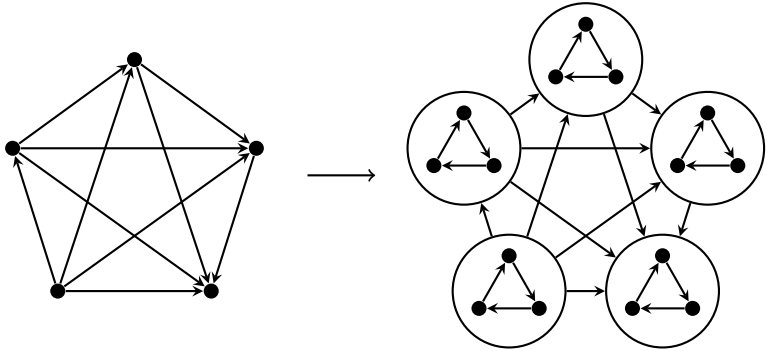

Take \( (i,u),(j,v) \in V(D) \) such that \( i<j \), then \( ((i,u),(j,v)) \in A(D) \).
Define \( D_i \), for every \( i \in [k] \), as the copy of \( \overrightarrow{C_3} \) in \( D \) corresponding to the vertex \( i \) of \( \overrightarrow{K_k} \).
The existence of a \( ((j,v),(i,u)) \)-path implies that there are \( (i',u'),(j',v') \in V(D) \) such that \( i'<j' \) and \( ((j',v'),(i',u')) \in A(D) \), thus contradicting the definition of \( \overrightarrow{K_k} \).
Moreover, that eliminates the possibility of a \( ((i,u),(i,v)) \)-path that is not contained in \( D_i \).
From this we conclude that for every \( S \subseteq V(D) \) we have \( I(S) = \bigcup_{i=1}^k I(S \cap V(D_i)) \) and \( I(S \cap V(D_i)) \subseteq V(D_i) \) for every \( i \in [k] \).
Therefore, \( [S] = V(D) \) if and only if \( [S \cap V(D_i)] = V(D_i) \) for every \( i \in [k] \), which means that \( \ohn (D) = \sum_{i=1}^k \ohn (D_i) = k\cdot \ohn (\overrightarrow{C_3}) =2k= \frac{2}{3} n(D) \).
\end{proof}

\section{Oriented Hull Number in Partial Cubes}
\label{section.bip}

In this section we prove that, given an oriented partial $D$ and a positive integer $k$, determining whether $\ohn (D) \leq k$ is an \NP-complete problem.

In the undirected case, it was proven that determining $\hn(G)$ is an \NP-hard problem, even if $G$ is a partial cube~\cite{CIPCTHN}.
Recall that if arcs in both ways were allowed, then this result would imply the \NP-hardness on the oriented case.
As we consider only oriented graphs, we first prove that replacing each edge with a directed $C_4$ has roughly the same effect in the class of bipartite graphs.

Given a (non-oriented) bipartite graph \( G \) with vertex set \( V(G) = \{ v_1, \ldots ,v_n \} \), let \( G_{\oracf} \) be the oriented bipartite graph such that
\( V(G_{\oracf}) = V(G) \cup \{ v_{i,j},v_{j,i} \mid v_iv_j \in E(G) \}\) and \[A(G_{\oracf}) = \{ (v_i,v_{i,j}),(v_{i,j},v_j),(v_j,v_{j,i}),(v_{j,i},v_i) \mid v_iv_j \in E(G) \}. \]
Thus, \(G_{\oracf}\) is obtained from $G$ by replacing each edge by a directed \( C_4 \).

The first important detail about this procedure is that it ``doubles'' the length of each path in \( G \).
In other words, if we have a path \( P =~v_{i_0}, v_{i_1}, \ldots ,v_{i_k}\) in \( G \), then \(G_{\oracf}\) has the directed paths:
\(P_1 = (v_{i_0}, v_{i_0,i_1}, v_{i_1}, \ldots ,v_{i_{k-1},i_k}, v_{i_k})\) and
$P_2 = (v_{i_k}, v_{i_k,i_{k-1}}, v_{i_{k-1}}, \ldots ,v_{i_1,i_0},v_{i_0}) $. We say that these three paths $P$, $P_1$ and $P_2$ are \textit{corresponding paths}.
Notice that for each \( v_iv_j \)-path of length \( d \) in \( G \) we have corresponding \( (v_i,v_j) \)-path and \( (v_j,v_i) \)-path in \( G_{\oracf} \), both of length \( 2d \).
Which means that \( 2d_G(v_i,v_j) = d_{G_{\oracf}}(v_i,v_j) = d_{G_{\oracf}}(v_j,v_i) \) for all \( v_i,v_j \in V(G) \).
Consequently, the corresponding paths of a geodesic in \( G \) are also geodesics in \( G_{\oracf} \).


Note that in Lemma~\ref{lemma.upboundGCfour} we do not require $G$ to be bipartite.

\begin{lem}
    \label{lemma.upboundGCfour}
    For any graph $G$, \( \ohn (G_{\oracf}) \leq \hn (G) \).
\end{lem}

\begin{proof}
    We prove that if $S$ is a hull set of \( G \), then $S$ is a hull set of \( G_{\oracf} \).
    Take a hull set \( S \) of \( G \).
    Since \( S \subseteq V(G) \), every geodesic considered in order to obtain \( I_{G_{\oracf}}(S) \) has a corresponding geodesic in \( G \), and the same is valid for \( I_G(S) \).
    We then have that \( I_{G} (S) = I_{G_{\oracf}} (S) \cap V(G) \).
    Thus, since there are vertices \( v_{i,j} \) in \( I_{G_{\oracf}} (S) \) we can state that \( I_{G}^n (S) \subseteq I_{G_{\oracf}}^n (S) \cap V(G) \) for every \( n \geq 2 \).
    We know that \( [S]_G \subseteq I_{G_{\oracf}}^n (S) \) for some natural \( n \) and that every \( v_{i,j} \) lies in the \( (v_i,v_j) \)-geodesic.
    Therefore, \( I_{G_{\oracf}}^{n+1} (S) \supseteq V(G_{\oracf}) \).
\end{proof}

Next we show the ``converse'' of the above statement: from any hull set of \( G_{\oracf} \) we can obtain a hull set of \( G \) with at most the same cardinality.
However, this fact is only valid for bipartite graphs.

\begin{lem}\label{lem.hngleqoracf}
    If \( G \) is a connected bipartite graph, then \( \hn (G) \leq \ohn (G_{\oracf}) \).
\end{lem}

\begin{proof}
    Let \( S \) be a hull set of \( G_{\oracf} \).
    We are going to construct another hull set \( S' \) of \( G_{\oracf} \) such that \( |S'| \leq |S| \) and \( S' \subseteq V(G) \).
    Because of the corresponding paths, we deduce that \( S' \) is a hull set of \( G \) and that the proposition holds.

    If \( S \subseteq V(G) \), then take \( S' = S \).
    Otherwise, denote \( S_0 := S \).
    We inductively construct \( S_d \) from \( S_{d-1} \), for \( d > 0 \), by replacing a vertex in \( S_{d-1} \setminus V(G) \) with one in \( V(G) \), preserving the hull set property.

    Let \( v_{i,j} \in S_{d-1} \setminus V(G) \).
    If there is a \( w \in S_{d-1} \setminus \{ v_{i,j} \} \) with some \( (v_{i,j},w) \)-geodesic \( \overrightarrow{P} = (v_{i,j},v_j,v_{j,i},v_i, \ldots ,w) \), notice that there is a \( (v_j,w) \)-geodesic containing \( v_i \) in \( G \).
    Defining \( S_d := (S_{d-1} \setminus \{ v_{i,j} \} ) \cup \{ v_j \} \) gives us \( v_i,v_j \in I(S) \), which implies \( v_{i,j} \in I^2(S) \).
    Thus \( S_{d-1} \subseteq I^2(S_d) \) and, consequently, \( V(G_{\oracf} ) \subseteq [S_{d-1}] \subseteq [S_d] \).
    If we have a \( (w,v_{i,j}) \)-geodesic \( ( w, \ldots ,v_j,v_{j,i},v_i,v_{i,j} ) \), we analogously take \( S_d = (S_{d-1} \setminus \{ v_{i,j} \}) \cup \{ v_i \} \).

    In case such a vertex does not exist in \( S_{d-1} \), then for every \( w \in S_{d-1} \setminus \{ v_{i,j} \} \) the \( (v_{i,j},w) \)-geodesics do not contain \( v_i \) and the \( (w,v_{i,j}) \)-geodesics do not contain \( v_j \).
    Take a vertex \( w \in S_{d-1} \setminus \{ v_{i,j} \} \) and a directed cycle \( C = G_{\oracf} [ \{ v_k,v_{k,l},v_l,v_{l,k} \} ] \) such that \( w \in V(C) \).
    Let \( P_1 \) and \( P_2 \) respectively be a \( (v_{i,j},w) \)-geodesic and a \( (w,v_{i,j}) \)-geodesic.
    We first prove by contradiction that \(  V(P_1) \cap V(P_2) = \{ v_{i,j},w \} \).
    Let \( w' \in V(P_1) \cap V(P_2) \setminus \{ v_{i,j},w \} \) so that the value of \( \min(v_{i,j},w') := \min \{ d(w',v_{i,j}),\) \( d(v_{i,j},w') \} \) is as small as possible.
    If there are \( r,s \) such that \( w' = v_{r,s} \) then we must also have \( v_r,v_s \in V(P_1) \cap V(P_2) \).
    Notice that the \( (v_{i,j},v_r) \)-path contained in \( P_1 \) and the \( (v_s,v_{i,j}) \)-path contained in \( P_2 \) are both geodesics, from where we get \( d(v_{i,j},v_r) = d(v_{i,j},w') - 1 \) and \( d(v_s,v_{i,j}) = d(w',v_{i,j}) - 1 \).
    Thus we have that \( \min (v_{i,j},v_r) \) or \( \min (v_{i,j},v_s) \) is smaller than \( \min (v_{i,j},w') \), contradicting the choice of \( w' \).
    Then we must have \( w' \in V(G) \).

    Observe that the \( (v_j,w') \)-geodesic \( P_1' \) contained in \( P_1 \) has an even number of arcs, \( q_1 \) of the form \( (v_r,v_{r,s}) \) and \( q_1 \) of the form \( (v_{r,s},v_s) \).
    The analogous is true for the \( (w',v_i) \)-geodesic \( P_2' \) contained in \( P_2 \), with length \( 2q_2 \).
    Now, take the \( (w',v_j) \)-path \( P \) such that \( V(P) = (V(P_1') \cap V(G) ) \cup \{ v_{s,r} \mid v_{r,s} \in V(P_1') \} \) and \( A(P) = \{ (v_s,v_{s,r}),(v_{s,r},v_r) \mid (v_r,v_{r,s}),(v_{r,s},v_s) \in A(P_1') \} \).
    Notice that we can extend \( P \) to obtain a \( (w',v_i) \)-path of length \( 2q_1 +2 \) containing \( v_j \), which is not a geodesic due to the case we are working on.
    Thus \( 2q_1+2 > 2q_2 \Rightarrow q_1 \geq q_2 \); analogously \( q_2 \geq q_1 \), giving us \( q_1 = q_2 \).
    We then have that \( P_1' \), \( P_2' \) and the arcs \( (v_i,v_{i,j}),(v_{i,j},v_j) \) together form a directed cycle of length \( 2q_1 + 2q_2 +2 = 2(2q_1 + 1) \).
    In \( G \) there is a corresponding undirected cycle with length \( 2q_1 + 1 \), contradicting the fact that \( G \) is bipartite.
    Since the existence of \( w' \) always results in a contradiction we must have \( V(P_1) \cap V(P_2) = \{ v_{i,j},w \} \).

    By the analysis made in the previous paragraph we cannot have \( w \in \{ v_k,v_l \} \).
    Thus assume without loss of generality that \( w = v_{k,l} \).
    Remember that we are working on the case in which there is no \( (w,v_{i,j}) \)-geodesic and no \( (v_{i,j},w) \)-geodesic containing both \( v_k,v_l \).
    Next we show that we can define \( S_d := ( S_{d-1} \setminus \{ v_{i,j},w \} ) \cup \{ v_i,v_k \} \).

    Let us analyse some \( (v_i,v_k) \)-geodesic \( P_3 \), which uses neither \( v_j \) nor \( v_l \) and has even length \( A(P_3) =: 2p \).
    Let \( P_1'' \) be the \( (v_j,v_k) \)-geodesic contained in \( P_1 \), which also has even length; notice that \( \big| |A(P_3)| - |A(P_1'')| \big| \) must be even.
    If \( |A(P_3)| = |A(P_1'')| \) we would have a directed cycle of length \( 4p + 2 \), which has a corresponding cycle in \( G \) with odd length.
    If \( |A(P_3)| \leq |A(P_1'')| - 2 \) then we could take a path \( P_3' = (v_{i,j},v_j,v_{j,i},v_i, \ldots ,v_k,v_{k,l}) \) containing \( P_3 \) with length \( |A(P_3')| = |A(P_3)| + 4 \leq |A(P_1'')| + 2 = |A(P_1)| \).
    Consequently it would be a \( (v_{i,j},v_{k,l}) \)-geodesic containing both \( v_i \) and \( v_j \), thus contradicting the hypothesis for this case.
    Therefore \( |A(P_3)| \geq |A(P_1'')| + 2 \).
    The argument is analogous for the \( (v_j,v_l) \)-geodesics not containing \( v_i,v_k \).
    Thus the \( (v_i,v_k) \)-path \( P_1' \) with \( V(P_1') = V(P_1'') \cup \{ v_i,v_{i,j} \} \) and \( A(P_1') = A(P_1'') \cup \{ (v_i,v_{i,j}),(v_{i,j},v_j) \} \) is a geodesic.
    Defining the \( (v_k,v_i) \)-path \( P_2' \) analogously we conclude that it is also a geodesic.
    We then have that denoting \( S_d := (S_{d-1} \setminus \{ v_{i,j},v_{k,l} \} ) \cup \{ v_i,v_k \} \) gives us \( v_{i,j},v_{k,l} \in V(P_1') \cup V(P_2') \subset I(S_d) \), which means that \( S_{d-1} \subset I(S_d) \) and consequently \( V(D) \subseteq [S_{d-1}] \subseteq [S_d] \).

    Following these steps inductively we obtain a hull set \( S' \subset V(G) \) of \( G_{\oracf} \).
    Notice that there might be \( v_{i,j},v_{k,l} \in S_d \) with \( \{ i,j \} \cap \{ k,l \} \neq \emptyset \) and that they were both exchanged for the same vertex of \( V(G) \).
    Now we only have left to prove that \( S' \) is also a hull set of \( G \).
    We show by induction on \( k \) that \( I_G^k(S') = I_{G_{\oracf}}^k (S') \cap V(G) \).

    Since \( S' \subset V(G) \), for all pairs of vertices \( v_i,v_j \in S' \) every \( (v_i,v_j) \)-geodesic has its corresponding \( v_iv_j \)-geodesic in \( G \).
    Thus it is easy to see that for \( k = 1 \) the claim follows.
    Now assume that it is true for \( k-1 \).
    Seen as \( I_G^{k-1}(S') = I_{G_{\oracf}}^{k-1}(S') \cap V(G) \), by the same argument used above it follows that \( I_G^k(S') = I( I_{G_{\oracf}}^{k-1}(S') \cap V(G) ) \cap V(G) \subseteq I_{G_{\oracf}}^k(S') \cap V(G) \).
    Next we show that \( I_{G_{\oracf}}^k(S') \cap V(G) \subseteq I_G^k(S') \).

    Taking \( w \in \left( I_{G_{\oracf}}^k(S') \cap V(G) \right) \setminus I_{G_{\oracf}}^{k-1}(S') \), we know that there are \( w_1,w_2 \in I_{G_{\oracf}}^{k-1}(S') \) such that \( w \) lies in some \( (w_1,w_2) \)-geodesic.
    For each \( \ell \in \{ 1,2 \} \), if \( w_\ell \) is a vertex of \( G \), we define \( w_\ell' := w_\ell \).
    Else suppose that there are distinct \( i_\ell,j_\ell \in \{ 1, \ldots ,n(G) \} \) such that \( w_\ell = v_{i_\ell,j_\ell} \).
    Since \( S' \subset V(G) \) let \( k' \in [k-1] \) be such that \( w_\ell \notin I_{G_{\oracf}}^{k'-1}(S') \) and \( w_\ell \in I_{G_{\oracf}}^{k'}(S') \).
    Thus there are \( w_{2\ell +1},w_{2 \ell + 2} \in I_{G_{\oracf}}^{k'-1}(S') \) such that \( w_\ell \) is interior to some \( (w_{2 \ell + 1},w_{2 \ell + 2}) \)-geodesic.
    We know that both indegree and outdegree of \( w_\ell \) are equal to one, which means that \( v_{i_\ell},v_{j_\ell} \in I_{G_{\oracf}}^{k'}(S') \).
    In this case we define \( w_1' := v_{j_1} \) and \( w_2' := v_{i_2} \).
    In any event we have a \( (w_1',w_2') \)-geodesic contained in the \( (w_1,w_2) \)-geodesic mentioned before.
    It is easy to see now that \( w \in I_G^k(S') \).
    This concludes our demonstration.
\end{proof}

By Lemmas~\ref{lemma.upboundGCfour} and~\ref{lem.hngleqoracf}, we deduce:

\begin{thm}
    \label{thm.npcbip}
    If $G$ is bipartite, then $\hn (G) = \ohn (G_{\oracf})$.
\end{thm}

To emphasize the necessity of \( D \) being bipartite in Lemma~\ref{lem.hngleqoracf}, we show a non-bipartite graph that violates the stated inequality.
Let \( G = C_3 \) with vertices \( v_1,v_2,v_3 \).
Since it is a complete graph, all its vertices are simplicial and therefore \( \hn (G) = 3 \).
Now considering \( G_{\oracf} \), notice that \( (v_{1,2},v_2,v_{2,3},v_3) \) and \( (v_3,v_{3,1},v_1,v_{1,2}) \) are both geodesics, thus \( v_1,v_2,v_3 \in I( \{ v_{1,2},v_3 \} ) \).
We then conclude that \( I^2( \{ v_{1,2},v_3 \} ) = V(G_{\oracf}) \), hence \( \ohn (G_{\oracf}) = 2 \).

It was proven in~\cite{OTHNOSGC} that determining whether \( \hn (G) \leq k \) is \NP-complete for a bipartite graph \( G \).
Thus, one can combine this result with Theorem~\ref{thm.npcbip} to deduce that, given an oriented bipartite graph $D$ and a positive integer $k$, deciding whether $\ohn (D) \leq k$ is an \NP-complete problem.
Moreover, one can observe that such reduction can also be applied to partial cubes, a subclass of bipartite graphs which we define in the sequel.


The \textit{hypercube graph of dimension n}, \( Q_n \), is a (undirected) graph such that its vertex set is \( V(Q_n) = \{ 0,1 \}^n \).
We express each vertex as \( v = (v^1, \ldots ,v^n) \), where \( v^i \in \{ 0,1 \} \) for every \( i \in \{ 1, \ldots ,n \} \).
The edge set of $Q_n$ is \( E(Q_n) = \{ uv \mid \) there is exactly one \( i \in \{ 1, \ldots ,n \} \) such that \( u^i \neq v^i \} \).
A \textit{partial cube} graph \( G \) is an isometric subgraph of some \( Q_n \), meaning that \( d_G(u,v) = d_{Q_n}(u,v) \) for every \( u,v \in V(G) \).
Moreover, since hypercubes are connected then the distance between each pair of vertices is positive and finite, hence partial cubes are also connected.

\begin{prop}
    \label{prop.partialcube}
    If $G$ is a partial cube, then $G_{\oracf}$ is an oriented partial cube.
\end{prop}

\begin{proof}
    Let \( G \) be a partial cube with \( V(G) = \{ v_1, \ldots ,v_l \} \).
    If \( k \) is the smallest positive integer such that \( G \subseteq Q_k \), each vertex of \( G \) can be considered as an element of \( \{ 0,1 \}^k \).
    Thus let \( H \subseteq Q_{2k} \) be a graph initially with \( V(H) = \{ u_1, \ldots ,u_l \} \subseteq \{ 0,1 \}^{2k} \) and, for each \( u_i = (u_i^1,u_i^2, \ldots ,u_i^{2k}) \) and \( v_i = (v_i^1,v_i^2, \ldots ,v_i^k) \), \( u_i^{2j-1} = u_i^{2j} = v_i^j \) for every \( j \in \{ 1, \ldots ,k \} \).
    Notice that each two vertices \( u_i \) and \( u_j \) differ in a positive even number of entries.

    Take two distinct vertices \( u_i,u_j \in V(H) \) such that there is only one \( m \in \{ 1, \ldots ,k \} \) so that \( u_i^{2m-1} = u_{i}^{2m} \neq u_j^{2m-1} = u_j^{2m} \) (notice that \( v_i,v_j \) are adjacent in \( G \)).
    For each such pair add the vertices \( u_{i,j} = (u_i^1, \ldots ,u_i^{2m-1},u_j^{2m},u_i^{2m+1}, \ldots ,\) \( u_i^{2k}) \) and \( u_{j,i} = (u_j^1, \ldots ,u_j^{2m-1},u_i^{2m},u_j^{2m+1}, \ldots ,u_j^{2k}) \) to \( V(H) \), which differ in the \( 2m-1^{\textrm{th}} \) and \( 2m^{\textrm{th}} \) entries.
    Notice also that both \( u_i \) and \( u_j \) differ in exactly one entry from both \( u_{i,j} \) and \( u_{j,i} \).
    Moreover add the edges \( u_iu_{i,j},u_{i,j}u_j,u_ju_{j,i},\) \(u_{j,i}u_i \) to \( E(H) \), which if oriented as \( (u_i,u_{i,j}),(u_{i,j},u_j),(u_j,u_{j,i}),(u_{j,i},u_i) \) result in \( G_{\oracf} \).

    Next we analyse each kind of pair of vertices of \( H \).
    If our pair is \( u_p,u_q \) we already know that they have at least two distinct entries.
    If it is \( u_p,u_{q,r} \) with \( p \in \{ q,r \} \) we also know that they have all but one entries in common.
    Now take \( p,q,r \in \{ 1, \ldots ,l \} \) such that \( u_{p,q} \in V(H) \) and \( r \notin \{ p,q \} \).
    Define \( m \colon \{ 1, \ldots ,l \}^2 \to \{ 0,1, \ldots ,k \} \) as \( m(i,j)=0 \) if either \( i=j \) or if \( u_i,u_j \) are not adjacent, and as the only \( m \in \{ 1, \ldots ,k \} \) such that \( u_i^{2m} \neq u_j^{2m} \) for every other pair \( (i,j) \).
    Since \( r \notin \{ p,q \} \), there must be \( m \in \{ 1, \ldots ,k \} \setminus \{ m(p,q) \} \) such that \( u_r^{2m-1} = u_r^{2m} \neq u_p^{2m-1} = u_p^{2m} \).
    We know that the \( 2m-1^{\textrm{th}} \) and the \( 2m^{\textrm{th}} \) entries of \( u_p \) and \( u_{p,q} \) are equal, thus \( u_r \) and \( u_{p,q} \) have at least two distinct entries.

    The last case is a pair \( u_{p,q},u_{r,s} \).
    If \( d_G(v_i,v_j) = 1 \) for every \( i \in \{ p,q \} \) and every \( j \in \{ r,s \} \), since \( v_p \) and \( v_q \) are adjacent in \( G \) by the existence of \( u_{p,q} \) we would have triangles in \( G \), which is not possible since it is bipartite.
    Thus there are \( i \in \{ p,q \} \) and \( j \in \{ r,s \} \) such that \( d_G(v_i,v_j) \geq 2 \), which means that we have \( m_1, \ldots ,m_t \in \{ 1, \ldots ,k \} \) such that \( v_i^m \neq v_j^m \) for every \( m \in \{ m_1, \ldots ,m_t \} \) and \( t = d_G(v_i,v_j) \).
    Suppose that there is some \( m \) in the previous set which is not in \( \{ m(p,q),m(r,s) \} \), then \( u_p^{2m-1} = u_p^{2m} = u_q^{2m-1} = u_q^{2m} \neq u_r^{2m-1} = u_r^{2m} = u_s^{2m-1} = u_s^{2m} \).
    Consequently, \( u_{p,q} \) and \( u_{r,s} \) have at least two distinct entries.
    Else we must have \( t=2 \), \( m_1 = m(p,q) \) and \( m_2 = m(r,s) \), without loss of generality.
    Notice that \( u_r^{2m_1-1} = u_r^{2m_1} = u_s^{2m_1-1} = u_s^{2m_1} \) and \( u_p^{2m_1-1} = u_p^{2m_1} \neq u_q^{2m_1-1} = u_q^{2m_1} \), which means that either the \( 2m_1-1^{\textrm{th}} \) or the \( 2m_1^{\textrm{th}} \) positions of \( u_{r,s} \) and \( u_{p,q} \) are different.
    The same can be said about the \( 2m_2-1^{\textrm{th}} \) and \( 2m_2^{\textrm{th}} \) positions, thus \( u_{p,q} \) and \( u_{r,s} \) have at least two distinct entries.
    Therefore, if we combine the results of these last two paragraphs with the construction of \( H \), we conclude that two vertices of \( H \) are adjacent if and only if they have only one distinct entry.

    That being said, all that is left to prove is that \( d_H(u,v) = d_{Q_{2k}}(u,v) \) for every \( u,v \in V(H) \).
    First consider \( u_i \) and \( u_j \).
    If \( v_i \) and \( v_j \) have \( p \) different entries, than their corresponding vertices in \( H \) differ in \( 2p \) entries implying that \( d_{Q_{2k}}(u_i,u_j) = 2p \).
    Since \( G \) is a partial cube we know that \( d_G(v_i,v_j) = d_{Q_k}(v_i,v_j) = p \), thus there is a \( v_iv_j \)-path with length \( p \) in \( G \).
    Due to the construction of \( H \), we also have in it a \( u_iu_j \)-path of length \( 2p \).
    Now suppose that there is a \( u_iu_j \)-path in \( H \) with length \( q < 2p \).
    Seen as the neighbors in \( H \) have exactly one distinct entry, then the endvertices of this path differ in at most \( q \) entries, which is a contradiction.
    Thus \( d_H(u_i,u_j) = 2p \).

    Next we analyze the distance between \( u_{i,j} \) and \( u_{r,s} \).
    Let \( \{ u,u' \} = \{ u_i,u_j \} \) and \( \{ v,v' \} = \{ u_r,u_s \} \) be such that \( u \) and \( v \) are as close as possible.
    It is straightforward that \( d_H(u_{i,j},u_{r,s}) = 2 + d_H(u,v) = 2(p+1) \), where \( 2p \) is the number of distinct entries of \( u \) and \( v \).
    Since we change one entry from \( u \) to \( u_{i,j} \) and one from \( v \) to \( u_{r,s} \) we must have \( 2p+2 \) distinct entries between \( u_{i,j} \) and \( u_{r,s} \), meaning that \( d_{Q_{2k}}(u_{i,j},u_{r,s}) = 2p+2 \).
    This also covers the proof that \( d_H(u_{i,j},u_r) = d_{Q_{2k}}(u_{i,j},u_r) \).
\end{proof}

\begin{cor}
    Given an oriented partial cube \( D \) and a positive integer \( k \), it is \NP-complete to decide whether \( \ohn(D) \leq k \).
\end{cor}
\begin{proof}
    Computing the hull number of undirected partial cubes is an \NP-hard problem~\cite{CIPCTHN}.
Thus, the proof is a direct consequence of Theorem~\ref{thm.npcbip} and Proposition~\ref{prop.partialcube}.
\end{proof}

\section{Reducing Set Cover to Oriented Geodetic Number}
\label{sec.setcover}

\newcommand\myprob[3]{%

\vspace{0.3cm}

\par\noindent #1\\

{\bfseries Input}: #2\\

{\bfseries Question}: #3\par

\vspace{0.3cm}

}

Our goal in this section is to prove that the following problem is \NP-complete, even if restricted to certain classes of oriented graphs:

\myprob{\textsc{Oriented Geodetic Number}}{Oriented graph $D$ and a positive integer $k$.}{Is $\ogn(D)\leq k?$}

First we point out that given a subset of vertices \( S\subseteq V(D) \), one can compute \( (u,v) \)-geodesics, for every \( u,v \in V(D) \), and decide whether \( S \) is a geodetic set in polynomial time similarly to the undirected case~\cite{TGHNIHFCG}.
Consequently, the problem is in \NP.
Thus, we need to prove the \NP-hardness.

The classes we work with are defined by the underlying graph \( G \) of \( D \), when that is either bipartite or split or cobipartite. A graph is \emph{split} (resp. cobipartite) if its vertex set can be partitioned into a stable set and a clique (resp. into two cliques). 
Notice that these graphs have a similar structure: the vertex set \( V(G) \) is partitioned into two smaller sets \( X \) and \( Y \), each one being either a clique or a stable set.
This general structure allows us to use a similar idea in each case to reduce the well-known \textsc{Set Cover}~\cite{Karp1972} problem to \textsc{Oriented Geodetic Number}.

\myprob{\textsc{Set Cover}}{\( U = \{ 1,2, \ldots ,n \} \), \( \mathcal{F} \subseteq \mathcal{P}(U) \) such that \( \bigcup_{F\in\mathcal{F}} F = U \) and $k\in\mathbb{Z}_+^*$.}{Does there exist \( \mathcal{F}' \subseteq \mathcal{F} \) such that \( \bigcup_{F \in \mathcal{F'}} F = U \) and $|\mathcal{F'}|\leq k$?}

The general ideia is to create an oriented graph \( D \) with a vertex set partition \( \{ A,B \} \). To each class, $A$ and $B$ will thus either be a stable set or a clique.
Given \( (U, \mathcal{F}, k) \) an instance to \textsc{Set Cover}, there are subsets \( X \subset A \) and \( Y \subset B \) corresponding to the sets \( \mathcal{F} \) and \( U \), respectively.
The vertices that do not lie in such subsets are auxiliary and they are either extreme or lie in a geodesic between the extreme vertices.

In each case, we shall prove that  \( (U, \mathcal{F}, k) \) is an YES-instance if and only if there is a geodetic set of \( D \) with at most \( k+3 \) vertices in \( X \) and none in \( Y \), and the vertices in \( X \) correspond to a desired family \( \mathcal{F}' \subset \mathcal{F} \) covering \( U \).

Since the reduction to each class has its particularities we make three individual proofs, but let us define a common subgraph to all the constructions. For an instance \(\mathcal{I} =  (U, \mathcal{F}, k) \) of \textsc{Set Cover}, let $D(\mathcal{I})$ be oriented bipartite graph whose vertex set is composed by two subsets of vertices $X$ and $Y$.
In $X$ there is one vertex $f_i$ corresponding to $F_i\in\mathcal{F}$, for every $i\in\{1,\ldots, m\}$.
In $Y$ there is a vertex $u_j$ corresponding to each element $j \in U$. Finally, for every $i\in\{1,\ldots, m\}$ and $j\in\{1,\ldots,n\}$, whenever $j \in F_i$ we have $(f_i,u_j) \in A(D)$.

\begin{thm}\label{thm.npgnbip}
    \textsc{Oriented Geodetic Number} is \NP-hard, even if $D$ has no directed cycle and its underlying graph is bipartite.
\end{thm}

\begin{proof}
    As we mentioned, we reduce the \textsc{Set Cover} problem to \textsc{Oriented Geodetic Number}. 
    Let $\mathcal{I}=(U= \{ 1,2, \ldots ,n \}, \mathcal{F} = \{F_1,\ldots, F_m\},k)$ be an input to \textsc{Set Cover}.
    We shall construct an oriented graph $D$ such that $(U, \mathcal{F},k)$ is an YES-instance if and only if $\ogn(D)\leq k+3$.

    We build $D$ from $D(\mathcal{I})$ by adding three vertices $u$, $v$ and $w$.
    Besides the arcs in $D(\mathcal{I})$, we add to $D$: $(u,f_i),(f_i,w)$ for every $i \in \{1,\ldots, m\}$, $(u_j,v)$ for every $j \in \{1,\ldots,n\}$, and finally $(u,v)$.

    By construction, $D$ is clearly a DAG whose underlying graph is bipartite (with partition \( V(D) = (X \cup \{ v \} ) \cup (Y \cup \{ u,w \} ) \)).
    Figure \ref{fig.scdag} exemplifies $D$ when $U = \{ 1,2,3,4,5 \}$ and $\mathcal{F} = \{ F_1 = \{ 1,2,3,4 \},F_2 = \{ 1,4 \},F_3 = \{ 2,3,5 \} \}$.
    Notice that $u$ is a source and that $v$ and $w$ are sinks, thus they belong to any geodetic set.
    Besides, \( (u,w) \notin A(D) \) and for every \( f_i \in X \) we have the geodesic \( (u,f_i,w) \).
    Since \( (u,v) \in A(D) \) we have \( I( \{ u,v,w \} ) = X \cup \{ u,v,w \} \).
    
    \begin{figure}[!htb]
        \begin{center}
            \begin{tikzpicture}
                \node[black] at (1.5,2.2) {\( u_1 \)};
                \node[black] at (1.5,1.3) {\( u_2 \)};
                \node[black] at (1.5,0.4) {\( u_3 \)};
                \node[black] at (1.5,-0.5) {\( u_4 \)};
                \node[black] at (1.5,-1.4) {\( u_5 \)};
                \node[black] at (1.5,-2.5) {\( Y \)};
                \node[black] at (-1.5,1.75) {\( f_1 \)};
                \node[black] at (-1.5,0.4) {\( f_2 \)};
                \node[black] at (-1.5,-0.95) {\( f_3 \)};
                \node[black] at (-1.5,-2.5) {\( X \)};
                \node[black] at (-3.9,0) {\( u \)};
                \node[black] at (3.9,0) {\( v \)};
                \node[black] at (0.4,-2.5) {\( w \)};
                \draw[black, thick, dashed] (1.5,0.15) ellipse (0.7cm and 2.3cm);
                \draw[black, thick, dashed] (-1.5,0.15) ellipse (0.7cm and 2.3cm);
                \node[circle,fill=black,inner sep=0pt,minimum size=0.3cm, name=1] at (1.5,1.8) {};
                \node[circle,fill=black,inner sep=0pt,minimum size=0.3cm, name=2] at (1.5,0.9) {};
                \node[circle,fill=black,inner sep=0pt,minimum size=0.3cm, name=3] at (1.5,0) {};
                \node[circle,fill=black,inner sep=0pt,minimum size=0.3cm, name=4] at (1.5,-0.9) {};
                \node[circle,fill=black,inner sep=0pt,minimum size=0.3cm, name=5] at (1.5,-1.8) {};
                \node[circle,fill=black,inner sep=0pt,minimum size=0.3cm, name=f1] at (-1.5,1.35) {};
                \node[circle,fill=black,inner sep=0pt,minimum size=0.3cm, name=f2] at (-1.5,0) {};
                \node[circle,fill=black,inner sep=0pt,minimum size=0.3cm, name=f3] at (-1.5,-1.35) {};
                \node[circle,fill=black,inner sep=0pt,minimum size=0.3cm, name=u] at (-3.5,0) {};
                \node[circle,fill=black,inner sep=0pt,minimum size=0.3cm, name=v] at (3.5,0) {};
                \node[circle,fill=black,inner sep=0pt,minimum size=0.3cm, name=w] at (0,-2.5) {};
                \draw[black, thick, ->, >=stealth] (-3.5,0) arc (180:3:3.5);
                \draw[black, thick, ->, >=stealth] (f1) -- (1);
                \draw[black, thick, ->, >=stealth] (f1) -- (2);
                \draw[black, thick, ->, >=stealth] (f1) -- (3);
                \draw[black, thick, ->, >=stealth] (f1) -- (4);
                \draw[black, thick, ->, >=stealth] (f2) -- (1);
                \draw[black, thick, ->, >=stealth] (f2) -- (4);
                \draw[black, thick, ->, >=stealth] (f3) -- (2);
                \draw[black, thick, ->, >=stealth] (f3) -- (3);
                \draw[black, thick, ->, >=stealth] (f3) -- (5);
                \draw[black, thick, ->, >=stealth] (u) -- (f1);
                \draw[black, thick, ->, >=stealth] (u) -- (f2);
                \draw[black, thick, ->, >=stealth] (u) -- (f3);
                \draw[black, thick, ->, >=stealth] (f1) -- (w);
                \draw[black, thick, ->, >=stealth] (f2) -- (w);
                \draw[black, thick, ->, >=stealth] (f3) -- (w);
                \draw[black, thick, ->, >=stealth] (1) -- (v);
                \draw[black, thick, ->, >=stealth] (2) -- (v);
                \draw[black, thick, ->, >=stealth] (3) -- (v);
                \draw[black, thick, ->, >=stealth] (4) -- (v);
                \draw[black, thick, ->, >=stealth] (5) -- (v);
            \end{tikzpicture}
            \caption{Oriented graph constructed in the proof of Theorem \ref{thm.npgnbip}, assuming \( U = \{ 1,2,3,4,5 \} \) and \( \mathcal{F} = \{ F_1 = \{ 1,2,3,4 \},F_2 = \{ 1,4 \},F_3 = \{ 2,3,5 \} \} \).}
            \label{fig.scdag}
        \end{center}
    \end{figure}
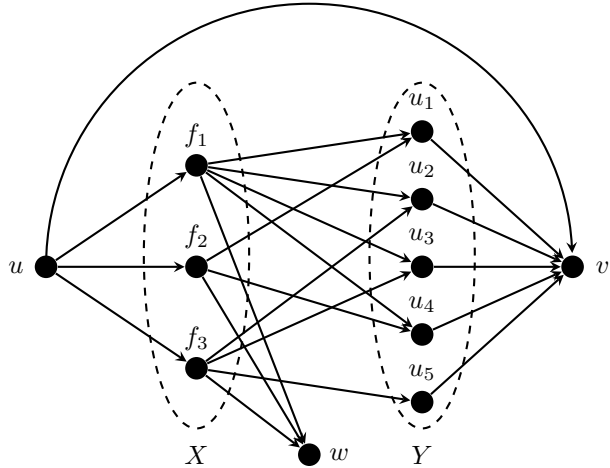

    Let $\mathcal{F}' = \{ F_i \mid i \in I\} \subseteq \mathcal{F}$, for some $I\subseteq \{1,\ldots, m\}$ such that $\bigcup_{i\in I} F_{i} = U$ and $|I| \leq k$.
    We then take $ \mathcal{X} = \{ f_i \mid i\in I \} \cup \{ u,v,w \} $, which has cardinality at most \( k+3 \).
    Thus, for every \( u_j \in Y \) there is \( f_i \in \mathcal{X} \) such that \( (f_i,u_j) \in A(D) \), from where we have the geodesic \( (f_i,u_j,v) \).
    Therefore, one can observe that $\mathcal{X}$ is a geodetic set of $D$.

    On the other hand, let $S$ be a geodetic set of $D$ with at most $k+3$ vertices.
    Thus, $\{u,v,w\} \subseteq S$ and at most $k$ vertices of $S$ belong to $X\cup Y$.
    If some $u_j \in S$, one can observe that by replacing $u_j$ with a vertex $f_i$ such that $(f_i,u_j)\in A(D)$, we obtain another geodetic set $S'$ such that $|S'|\leq k+3$.
    Thus, without loss of generality, we assume that $S \setminus\{u,v,w\}\subseteq X$.
    Let $I = \{ i \in \{1,\ldots,m\} \mid f_i \in S \}$.
    One can also observe that the family $ \mathcal{F'} = \{ F_i \in \mathcal{F} \mid i \in I \}$ satisfies $\bigcup_{F_i \in \mathcal{F'}} F_i = U$ and $|\mathcal{F'}| \leq k$.
\end{proof}

The particularities we mentioned before are related to the basic difference of each class: the vertex partition \( \{ A,B \} \).
In Theorem~\ref{thm.npgnbip} we assumed that both \( A \) and \( B \) were stable.
In Theorem~\ref{thm.npgnsplit}, we address to the case in which the underlying graph is split and then we consider \( A \) as a clique, while \(B\) is still a stable set.

\begin{thm}\label{thm.npgnsplit}
    \textsc{Oriented Geodetic Number} is \NP-hard, even if $D$ has an underlying split graph.
\end{thm}

\begin{proof}
    We reduce the \textsc{Set Cover} problem to \textsc{Oriented Geodetic Number}.
    Let $\mathcal{I}=(U= \{ 1,2, \ldots ,n \}, \mathcal{F} = \{F_1,\ldots, F_m\},k)$ be an input to \textsc{Set Cover}.
    We shall construct an oriented graph $D$ such that $(U, \mathcal{F},k)$ is a YES-instance if and only if $\ogn(D) \leq k+3$.

    We build \( D\) from \(D(\mathcal{I})\) by adding four vertices \( u \), \( w \), \( x \) and \( y \). Besides the arcs in \(D(\mathcal{I})\) our oriented split graph \( D \) has the following arcs: $(f_i,f_j)$ whenever $1\leq i<j\leq m$; \( (u,f_i),(x,f_i),(f_i,w) \) for every \( i \in \{ 1, \ldots ,m \} \); \( (u_j,x) \) for every \( j \in \{ 1, \ldots ,n \} \); and finally \( (u,x) \) and \( (x,y) \).

    Notice that by construction, \( C = \{ u,x \} \cup X \) is a clique and \( S = \{ w,y \} \cup Y \) is a stable set, therefore the underlying graph of \( D \) is split.
    Moreover, \( u \) is a source and \( w,y \) are sinks, thus they belong to any geodetic set.
    If we analyse the \( (u,w) \)-geodesics and the \( (u,y) \)-geodesics we find that they are respectively \( (u,f_i,w) \), for every \( i \in \{ 1, \ldots ,m \} \), and \( (u,x,y) \).
    We then have \( I( \{ u,w,y \} ) = C \cup \{ w,y \} \).

    Let \( \mathcal{F}' = \{ F_i \mid i \in I\} \subseteq \mathcal{F} \), for some $I \subseteq \{1, \ldots, m\}$ such that \( \bigcup_{i\in I} F_{i} = U \) and \( |I| \leq k \).
    We then take $ \mathcal{X} = \{ f_i \mid i\in I \} \cup \{ u,w,y \} $, which has cardinality at most \( k+3 \).
    Thus, for every \( u_j \in Y \) there is \( f_i \in \mathcal{X} \) such that \( (f_i,u_j) \in A(D) \).
    Since \( (x,f_i) \in A(D) \) for every \( i \in \{ 1, \ldots ,m \} \), \( (f_i,u_j,x,y) \) is a geodesic for every \( i \in I \) and every \( j \in \{ 1, \ldots ,n \} \).
    Therefore, one can observe that $\mathcal{X}$ is a geodetic set of $D$.

    On the other hand, let $\mathcal{X}$ be a geodetic set of $D$ with at most $k+3$ vertices.
    Thus, $\{ u,w,y \} \subseteq \mathcal{X}$ and at most $k$ vertices of $\mathcal{X}$ belong to $X \cup Y \cup \{ x \}$.
    If some $ u_j \in \mathcal{X}$, one can observe that by replacing $u_j$ with a vertex $f_i$ such that $(f_i,u_j) \in A(D)$, we obtain another geodetic set $\mathcal{X}'$ such that $|\mathcal{X}'| \leq k+3$.
    Moreover, taking an arbitrary vertex \( v \in V(D) \), any \( (v,x) \)-geodesic can be extended to a \( (v,y) \)-geodesic, since \( N^+(x) = \{ y \} \).
    Thus, without loss of generality, we assume that $ \mathcal{X} \setminus \{ u,w,y \} \subseteq X$.
    Let $I = \{i \in \{ 1, \ldots ,m \} \mid u_i \in S\}$.
    One can also observe that the family $\mathcal{F'} = \{ F_i \in \mathcal{F} \mid i \in I \}$ satisfies $ \bigcup_{F_i \in \mathcal{F'}} F_i = U$ and $|\mathcal{F'}| \leq k$.
\end{proof}

In Theorem~\ref{thm.npgncobip} we assume that both \( A \) and \( B \) are cliques.
The oriented graph \( D \) that we shall build is very similar to the one used in the proof of Theorem~\ref{thm.npgnbip}, we only add more arcs because we need $A$ and $B$ to be cliques instead of stable sets.

\begin{thm}\label{thm.npgncobip}
    \textsc{Oriented Geodetic Number} is \NP-hard, even if $D$ has no directed cycle and its underlying is cobipartite.
\end{thm}

\begin{proof}
    We reduce the \textsc{Set Cover} problem to \textsc{Oriented Geodetic Number}. 
    Let $\mathcal{I} = (U= \{ 1,2, \ldots ,n \}, \mathcal{F} = \{F_1,\ldots, F_m\},k)$ be an input to \textsc{Set Cover}.
    We shall construct an oriented graph $D$ such that $(U, \mathcal{F},k)$ is a YES-instance if and only if $\ogn(D)\leq k+3$.

    We build $D$ from \(D(\mathcal{I})\) by adding three vertices $u$, $v$ and $w$. Besides the arcs in \(D(\mathcal{I})\), \(D\) has the following arcs: $(f_i,f_j)$ whenever $1\leq i<j\leq m$; $(u_i,u_j)$ whenever $1\leq i<j\leq n$; $(u,f_i),(f_i,w)$ for every $i \in \{1,\ldots, m\}$; $(u,u_j),(u_j,v)$ for every $j \in \{1,\ldots,n\}$; and finally $(u,v)$.

    By construction, both \( C_1 = X \cup \{ w \} \) and \( C_2 = Y \cup \{ u,v \} \) are cliques, thus \( D \) is an oriented cobipartite graph. Moreover, \((u,f_1,\ldots, f_m, u_1,\ldots,u_n,w,v)\) is a total ordering of \(V(D)\) so that no arc is oriented backwards and thus \(D\) is an acyclic orientation.
    Notice that $u$ is a source and that $v$ and $w$ are sinks, thus they belong to any geodetic set.
    Besides, \( (u,w) \notin A(D) \) and for every \( f_i \in X \) we have the geodesic \( (u,f_i,w) \).
    Since \( (u,v) \in A(D) \) we have \( I( \{ u,v,w \} ) = X \cup \{ u,v,w \} \).

    Let $\mathcal{F}' = \{ F_i \mid i \in I\} \subseteq \mathcal{F}$, for some $I \subseteq \{1,\ldots, m\}$ such that $\bigcup_{i\in I} F_{i} = U$ and $|I| \leq k$.
    We then take $ \mathcal{X} = \{ f_i \mid i \in I \} \cup \{ u,v,w \} $, which has cardinality at most \( k+3 \).
    Thus, for every \( u_j \in Y \) there is \( f_i \in \mathcal{X} \) such that \( (f_i,u_j) \in A(D) \), from where we have the geodesic \( (f_i,u_j,v) \).
    Therefore, one can observe that $\mathcal{X}$ is a geodetic set of $D$.

    On the other hand, let $\mathcal{X}$ be a geodetic set of $D$ with at most $k+3$ vertices.
    Thus, $\{u,v,w\} \subseteq \mathcal{X}$ and at most $k$ vertices of $\mathcal{X}$ belong to $X\cup Y$.
    If some $u_j \in S$, one can observe that by replacing $u_j$ with a vertex $f_i$ such that $(f_i,u_j) \in A(D)$, we obtain another geodetic set $\mathcal{X}'$ such that $|\mathcal{X}'| \leq k+3$.
    Thus, without loss of generality, we assume that $\mathcal{X} \setminus \{ u,v,w \} \subseteq X$.
    Let $I = \{ i \in \{ 1, \ldots ,m \} \mid f_i \in \mathcal{X} \}$.
    One can also observe that the family $ \mathcal{F'} = \{ F_i \in \mathcal{F} \mid i \in I \}$ satisfies $\bigcup_{F_i \in \mathcal{F'}} F_i = U$ and $|\mathcal{F'}| \leq k$.
\end{proof}

\begin{cor}
    \textsc{Oriented Geodetic Number} is \W[2]-hard parameterized by $k$ and does not admit any polynomial-time approximation algorithm with ratio $c \cdot \ln n$, unless $\Poly = \NP$, whenever $G$ is split, bipartite or cobipartite.
\end{cor}
\begin{proof}
    One should notice that in the proofs of Theorems~\ref{thm.npgnbip},~\ref{thm.npgnsplit} and~\ref{thm.npgncobip} the constructed instances have linear size in $n+m$ (recall that $|U|=n$ and $|\mathcal{F}|=m$) and the parameter in the constructed instance of \textsc{Oriented Geodetic Number} is $k+3$, where $k$ is the one provided by the the input instance of \textsc{Set Cover}. Thus, as these hardness results hold for the Set Cover problem~\cite{RS1997,DF2012}, they also do for \textsc{Oriented Geodetic Number}.
\end{proof}

\section{Polynomial-time algorithm for cacti}
\label{section.tc}

The hull and geodetic numbers of an undirected tree \( T \) are both equal to the number of leaves of \( T \).
Notice that the leaves of a tree are its simplicial vertices.
Moreover, any node belongs to a \( uv \)-path for some distinct leaves $u,v \in V(T)$, and that path is a geodesic because it is unique.
That means that the set of simplicial vertices of a tree is a minimum hull and geodetic set.
A similar statement is true for the oriented case.

\begin{prop}
    Let \( D \) be an oriented tree.
    Then, $\ext(D)$ is both a minimum hull set and a minimum geodetic set of $D$.
    Consequently, they are unique and \( \ohn (D) = \ogn (D) = |\ext(D)| \).
\end{prop}

\begin{proof}
    Let \( u \in V(D) \) be a non-extreme vertex and \( P = (v_1, \ldots ,u, \ldots ,v_2) \) a maximal (directed) path with \( u \) as an internal vertex.
    Due to the maximality of \( P \), we must have either \( d^-(v_1) = 0 \) or \( N^-(v_1) \subset V(P) \).
    Since the second alternative would imply that \( D \) has a cycle, which cannot happen in a tree, we must have \( d^-(v_1) = 0 \).
    Analogously, \( d^+(v_2) = 0 \).
    Thus, both \( v_1 \) and \( v_2 \) are extreme.
    Moreover, we know that there is only one \( v_1v_2 \)-path in the underlying tree of \( D \), which means that \( P \) is a geodesic.
    Therefore, \( I( \ext (D) ) = V(D) \).
\end{proof}

Thus, one can also observe that both the hull and the geodetic numbers of oriented trees, a subclass of oriented bipartite graphs, can be computed in linear time.

This result led us to work on the cacti, a superclass of trees.
A graph is called a \textit{cactus} if each block is either an edge or a cycle.
Consequently, every cycle is an induced cycle and two cycles intersect in at most one vertex.

For cacti, there are algorithms to compute the geodetic and hull numbers of a non-oriented cactus graph proposed in~\cite{CMGSPIG, OTHNOSGC}. As in the undirected case, the extreme vertices may not suffice in order to obtain a hull or a geodetic set of an oriented cactus.
It is necessary to include a few non-extreme vertices of some particular cycles.
We introduce below a few important notions in order to define such cycles.

In the remainder of this section, let $D$ be an oriented cactus graph.
Let \( C \subseteq D \) be a (oriented) cycle and \( u \in V(C) \) be a cut-vertex of $D$.
If there is an arc \( (u,v) \in A(D) \setminus A(C) \), we say that \( u \) is a \textit{transmitter} cut-vertex of \( C \) and use the initials TCV.
Analogously, if \( (v,u) \in A(D) \setminus A(C) \) we say that \( u \) is a \textit{receiver} cut-vertex and use the initials RCV.
Notice that we can have a cut-vertex which is both an RCV and a TCV.

A cycle \( C \subseteq D \) is called a \textit{leaf cycle} if it has only one cut-vertex.
We say that a \textit{trap cycle} is a directed cycle \( C \) such that its cut-vertices are either all transmitters and all not receivers; or they are all receivers, but not transmitters. In order words, $C$ is a trap cycle if it is a directed cycle and all edges with one endpoint in $C$ and the other in $D\setminus C$ are oriented the same way - either towards $C$ or towards $D\setminus C$. 
If \( C \) is a trap cycle with a TCV we say that it is a \textit{transmitter} trap cycle.
Similarly, a trap cycle with an RCV is a \textit{receiver} trap cycle.
At last, we say that a cycle \( C \) is \textit{unsatisfactory} if one of the following holds:
\begin{enumerate}[\textrm{Type} 1:]
    \item \( C \) is a trap cycle;
    \item \( C \) is a directed leaf cycle that is not a trap cycle;
    \item there are only two vertices in \( \ext(C) \), say \( u_1 \) e \( u_2 \), such that the two \( (u_1,u_2) \)-paths in \( C \) have different lengths and the longest one does not have internal cut vertices.
\end{enumerate}
When none of these happens, we say that the cycle is \textit{satisfactory}.
To shorten our text we also use the acronym UCi for ``unsatisfactory cycle of type i'', with \( i \in \{ 1,2,3 \} \).

Recall that any hull set must intersect any co-convex set.
Next we show that each unsatisfactory cycle contains a co-convex set and we describe these sets for each type of cycle.

\begin{lem}\label{lemma.unsatisfactory}
    Let \( D \) be an oriented cactus graph and \( C \) be an unsatisfactory cycle of $D$. Then, \( V(C) \) contains a co-convex set \( S \) and if \( C \) is of type:
    \begin{itemize}
        \item 1, then \( S = V(C) \);
        \item 2, then \( S = V(C) \setminus \{ w \} \) where \( w \) is the cut-vertex of \( C \);
        \item 3 with \( \ext(C) = \{ u,v \} \), then \( S \) consists of the internal vertices of the longest \( (u,v) \)-path.
    \end{itemize}
    Moreover, there is no intersection between any co-convex sets of different cycles.
\end{lem}

\begin{proof}
    First let \( C \) be a receiver trap cycle, without loss of generality.
    Suppose that there are distinct vertices \( u_1,u_2 \in N(V(C)) \) and a \( (u_1,u_2) \)-geodesic \( P \) such that all its internal vertices are in \( V(C) \).
    We then have \( v \in V(C) \) such that \( (u_1,v) \in A(P) \), contradicting the choice of \( C \).

    Now let \( C \) be a UC2 and let \( w \in V(C) \) be its cut-vertex in \( D \).
    Thus, if there is a path \( P \) as the one suggested in the previous paragraph then \( w \) is the only vertex in \( V(C) \cap V(P) \).

    At last, let \( C \) be a UC3 and let \( v_1,v_2 \) be the extreme vertices in \( C \), respectively the source and the sink.
    Denote by \( P_1 \) and \( P_2 \) respectively the shortest and the longest \( (v_1,v_2) \)-paths in \( C \); we prove next that \( V(P_2) \setminus \{ v_1,v_2 \} \) is a co-convex set.
    Now take \( u_1,u_2 \in N(V(P_2) \setminus \{ v_1,v_2 \}) \) and \( P \) a \( (u_1,u_2) \)-path as the one in the first paragraph.
    By the definition of \( C \) we have that \( u_1 = v_1 \) and \( u_2 = v_2 \), contradicting the choice of \( u_1 \) and \( u_2 \).
    Since that is not a geodesic, the result follows.

    We conclude this demonstration by proving that these co-convex sets do not intersect.
    Two trap cycles cannot intersect because if they did, the common vertex would be an RCV and a TCV for both cycles.
    If one of the cycles is not a trap cycle then there is no intersection, since its co-convex set does not contain a cut-vertex.
\end{proof}

We have proven the necessity of having at least one vertex from each unsatisfactory cycle in any hull set.
Now for the main result, we have left to show that these along with \( \ext (D) \) are enough.

\begin{lem}\label{lemma.paths}
    Let \( D \) be an oriented cactus and \( u \in V(D) \).
    For every \( v \in N^+(u) \) (\( v \in N^-(u) \)) there is a maximal path \( P = (v_1, \ldots ,v_q) \) with \( u=v_1 \) and \( v = v_2 \) (\( u=v_q \) and \( v = v_{q-1} \)) such that \( v_q \) (\( v_1 \)) either is extreme or belongs to an unsatisfactory cycle \( C \) of type either 1 or 2.
    Moreover, all the vertices of \( C \) are in \( V(P) \).
\end{lem}

\begin{proof}
    Consider all of the maximal paths (starting in \( u \)) of the form \( P = (u,v_1 = v,v_2, \ldots ,v_q) \).
    We have two options: either \( N^+(v_q) = \emptyset \) or \( N^+(v_q) \subseteq V(P) \).
    If the first one occurs then \( v_q \) is extreme.
    Else we have \( N^+(v_q) \subseteq V(P) \), and since \( D \) is a cactus we must also have \( d^+(v_q) = 1 \).
    Moreover we have a directed cycle \( C \) whose vertices are all in \( V(P) \).
    If \( C \) is either a UC1 or a UC2 then we are done.
    Also notice that \( C \) cannot be a UC3, since it is directed.

    Assume that every maximal path as described above ends in a satisfactory cycle \( C \).
    Take \( P \) as a maximal path which intersects the largest number of cycles in \( D \) and let \( C_1, \ldots ,C_n \) be those cycles such that \( C_i \) is the \( i^{\textrm{th}} \) cycle that \( P \) intersects (with respect to its orientation).
    Notice that all of these cycles are pairwise different, because if not we would have a block which would neither be a vertex nor a cycle.
    Since \( C_n \) is satisfactory and directed, there is an arc \( (w_1,w_2) \in A(D) \) with \( w_1 \in V(C_n) \setminus \{ v_q \} \) and \( w_2 \notin V(C_n) \).
    Therefore we can take another maximal path \( P^* = (u,v,v_2, \ldots ,w_1,w_2, \ldots ,v_{q'}) \) ending in a satisfactory directed cycle such that \( P \cap P^* = (u,v,v_2, \ldots ,w_1) \).
    It is easy to see that \( P^* \) intersects at least one more cycle than \( P \), thus contradicting the choice of the latter.

    For the maximal paths \( (v_0,v_1, \ldots ,v,u) \) the argument is analogous.
\end{proof}

With that we can obtain a path between two vertices in any hull set with \( v \) as an internal vertex, but it does not guarantee the existence of a geodesic.
However, for some vertices if there is a geodesic not containing \( v \) then that would create a block that is not allowed in a cactus graph.
Next we present some restrictions for these paths.

\begin{lem}\label{lemma.pathsparts}
    Let \( D \) be an oriented cactus graph with \( C \subseteq D \) a cycle, \( u,v \in V(D) \) and \( P\subseteq D \) a \( (u,v) \)-path.
    \begin{enumerate}
        \item If there is a \( w \in V(P) \) that does not belong to any cycle of \( D \) then every \( (u,v) \)-path contains \( w \).
        \item Let \( (w_{1}, \ldots ,w_{q}) = P \cap C \) with \( q \geq 2 \). If $P^*$ is a \( (u,v) \)-path in $D$, then $P^*\cap C$ is a $(w_1,w_q)$-path.
    \end{enumerate}
\end{lem}

\begin{proof}
    In the first case, by contradiction suppose that there exists a \( (u,v) \)-path \( P^* \) that does not contain \( w \).
    Let \( v_1,v_2, \ldots ,v_r \) be the vertices which are in both \( P \) and \( P^* \), ordered according to the orientation of \( P \).
    Take \( i \in \{ 1, \ldots ,r-1 \} \) such that \( w \) lies in the \( (v_i,v_{i+1}) \)-path contained in \( P \).
    The two paths between \( v_i \) and \( v_{i+1} \) (the ones contained in \( P \) and in \( P^* \)) are internally disjoint, thus together they form a cycle.
    This contradicts the choice of \( w \).

    For the second, also by contradiction assume first that there is another \( (u,v) \)-path \( P^* \subset D \) that does not intersect \( C \).
    Using arguments analogous to the ones above we conclude that there is a block containing \( C \) which is not a cycle.

    Now suppose that there is an \( (u,v) \)-path \( P^* \subset D \) intersecting \( C \), with \( P'' := P^* \cap C \) such that its first vertex is \( w_1' \neq w_1 \).
    Since \( u \in V(P) \cap V(P^*) \) let \( u' \) be the last vertex (with respect to the orientation of \( P \)) in that intersection before the cycle \( C \).
    Therefore \( C \), the \( (u',w_1) \)-path contained in \( P \) and the \( (u',w_1') \)-path contained in \( P^* \) are in the same block, which again is a contradiction to the fact that \( D \) is a cactus.
    The argument for \( w_q \) in the statement of this lemma is analogous.
\end{proof}

Combining the arguments provided by Lemmas~\ref{lemma.unsatisfactory},~\ref{lemma.paths} and~\ref{lemma.pathsparts}, one may deduce how to obtain a minimum hull set of an oriented cactus.
In the proof of Theorem~\ref{teo.hcactus} we analyse each type of cycle, showing which vertices of the unsatisfactory ones to take for the hull set.
Before we proceed, it is important to explain how the extreme vertices appear on cycles.

\begin{lem}\label{lemma.extcycles}
    Let \( C \) be an oriented cycle.
    \begin{enumerate}[1.]
        \item if \( |V(C)| = 3 \), then either \( \ext (C) = V(C) \) and it has one extreme vertex of each kind or it is a directed cycle without extreme vertices;
        \item if \( |V(C)| \geq 4 \), it has no transitive vertices and the number of sources is equal to the number of sinks.
    \end{enumerate}

    \begin{proof}
        First let \( C \) be a cycle with three vertices \( u,v,w \).
        If \( u \) is transitive then \( (v,u),(u,w),(v,w) \in A(C) \) and thus \( v \) is a source and \( w \) is a sink.
        If \( u \) is a source then \( (u,v),(u,w) \in A(C) \); consequently, it is straightforward that the other two vertices will be a sink and a transitive.
        The argument is analogous for the case where \( u \) is a sink.
        Last we assume that \( u \) is not extreme, where we have \( (w,u),(u,v),(v,w) \in A(C) \) without loss of generality; notice that \( C \) is a directed cycle and has no extreme vertices.

        Now we work on the case where \( |V(C)| \geq 4 \).
        If there is a transitive vertex \( v \) in \( V(C) \), then there are at least \( u,w \in V(C) \) such that \( (u,v),(v,w),(u,w) \in A(C) \).
        But then \( C \) would only have these three vertices, contradicting our assumption; therefore, \( C \) has no transitive vertices.
        Notice now that the number of arcs in \( C \) is equal to the sum of the indegrees of the vertices and to the sum of the outdegrees of the vertices.
        For a non-extreme vertex of \( C \), both the indegree and the outdegree are one; for a source, the indegree is zero and the outdegree is two; and for a sink, the indegree is two and the outdegree is zero.
        Hence we conclude that the number of sources is equal to the number of sinks in \( C \).
    \end{proof}
\end{lem}

\begin{thm}\label{teo.hcactus}
    Let $D$ be an oriented cactus graph.
    Then there exists a minimum hull set $S$ of $D$ composed by the extreme vertices of $D$ and by exactly one non-extreme vertex of each unsatisfactory cycle.
    Moreover, $I(S)$ contains all the vertices that are not in a satisfactory cycle and $I^2(S) = V(D)$.
\end{thm}

\begin{proof}
    Besides the extreme vertices, by Lemma \ref{lemma.unsatisfactory} a hull set of \( D \) must also contain at least one non-extreme from each unsatisfactory cycle.
    Next we show how to obtain a hull set \( S \) with only these vertices, which consequentially will be minimum.

    Let \( C \subset D \) be a receiver trap cycle (the argument for transmitter trap cycles is analogous).
    Take \( u,u' \in V(C) \) such that \( (u,u') \in A(C) \) and \( u' \) is a cut-vertex, which means that there is \( w \in N^-(u') \setminus V(C) \).
    By Lemma \ref{lemma.paths}, there is a path \( P = (v, \ldots ,w,u') \) such that \( V(C) \cap V(P) = u' \) and \( v \) is either extreme or lies in a cycle \( C' \), which is either an UC1 or an UC2, with \( V(C') \subset V(P) \).
    If \( v \) is extreme then call it \( v' \) from now on.
    If the second case happens, by Lemma \ref{lemma.unsatisfactory} we must have a vertex of \( C' \) in every hull set; let \( v' \) be that vertex and let \( P' \subseteq P \) be the \( (v',u') \)-path.
    In both events, we can extend \( P '\) to get a \( (v',u) \)-path which contains every vertex of \( C \) since it is directed and \( (u,u') \in A(D) \).
    By Lemma \ref{lemma.pathsparts} every \( (v',u) \)-path arrives at \( C \) by the vertex \( u' \), including the geodesics.
    Therefore, by adding \( u \) to \( S \) we have \( V(C) \subseteq I(S) \).

    Now let \( C \) be an UC2 and let \( v \in V(C) \) be its cut-vertex.
    Since \( C \) is not a trap cycle there are \( v_1,v_2 \in V(D) \setminus V(C) \) such that \( (v_1,v),(v,v_2) \in A(D) \).
    Taking an arbitrary \( u \in V(C) \setminus \{ v \} \), by Lemma \ref{lemma.paths} there are \( w_1,w_2 \) in the hull set \( S \) such that we have a \( (w_1,u) \)-path and an \( (u,w_2) \)-path.
    Notice that each vertex of \( C \) is in at least one of these paths.
    Moreover by Lemma \ref{lemma.pathsparts} we deduce that each vertex of \( C \) other than \( u \) and \( v \) lies either in a \( (w_1,u) \)-geodesic or in an \( (u,w_2) \)-geodesic.
    Therefore adding \( u \) to \( S \) results in \( V(C) \subseteq I(S) \).

    Next we analyse the case where \( C \) is an UC3, such that \( u_1 \in V(C) \) is a source in \( C \) and \( u_2 \in V(C) \) is a sink in \( C \).
    Take \( v \) an arbitrary interior vertex of the longest \( (u_1,u_2) \)-path in \( C \).
    If \( u_1 \) is also an extreme vertex in \( D \) we take \( w_1 = u_1 \).
    If not then we have \( u_1' \in N^-(u_1) \setminus V(C) \).
    By Lemma \ref{lemma.paths} there are a vertex \( w_1 \) either extreme or lying in an UC1 or in an UC2 and a path \( (w_1, \ldots ,u_1',u_1) \).
    Define \( w_2 \) analogously.
    Thus we have an \( (w_1,w_2) \)-path intersecting \( C \) (that intersection being an \( (u_1,u_2) \)-path), a \( (w_1,u) \)-path and an \( (u,w_2) \)-path.
    By Lemma \ref{lemma.pathsparts} every \( (w_1,w_2) \)-geodesic contains the \( (u_1,u_2) \)-geodesic and we also have a \( (w_1,u) \)-geodesic and an \( (u,w_2) \)-geodesic.
    Therefore \( V(C) \subseteq I(S) \).

    Take a vertex \( v \) which is not in any cycle and neither is extreme.
    Thus its indegree and outdegree are positive.
    By the Lemmas \ref{lemma.paths} and \ref{lemma.pathsparts} there is an \( (u,w) \)-geodesic containing \( v \) with \( u,w \in S \).
    We then have \( v \in I(S) \).

    This leaves the satisfactory cycles to be analyzed.
    First let \( C \) be a cycle with only three vertices \( u,v,w \).
    By Lemma \ref{lemma.extcycles}, we have two possibilities.
    If it is directed, since it is neither a trap nor a leaf cycle there are distinct RCV and TCV, suppose \( u \) and \( v \) respectively.
    By Lemmas \ref{lemma.paths} and \ref{lemma.pathsparts} there are \( w_1,w_2 \in S \) such that a \( (w_1,w_2) \)-geodesic contains the \( (u,v) \)-geodesic.
    If \( w \) is internal to the latter, then \( V(C) \subseteq I(S) \); else it is internal to the \( (v,u) \)-path, and since \( u,v \in I(S) \) we have \( V(C) \in I^2(S) \).
    Now assume that \( u \) is a source, \( w \) is a sink and \( v \) is transitive.
    By the same Lemmas mentioned before, we have \( u,w \in I(S) \).
    If \( v \) is also extreme in \( D \) then \( w \subseteq S \), else it is a cut-vertex of \( D \) and we can use the pair of lemmas to conclude that \( w \in I(S) \).

    From now on \( C \) will be a cycle with at least four vertices.
    If it has at least four extreme vertices, take an arbitrary vertex \( v \) of \( C \).
    Let \( P \) be a maximal path in \( C \) containing \( v \), then it is an \( (u_1,u_2) \)-path with \( u_1,u_2 \) respectively being a source and a sink in \( C \).
    Define \( w_1 \) and \( w_2 \) like in the case where \( C \) is an UC3.
    Then by the Lemmas \ref{lemma.paths} and \ref{lemma.pathsparts} we have a \( (w_1,w_2) \)-geodesic containing the \( (u_1,u_2) \)-path in \( C \).
    Applying this argument to every vertex of \( C \), we deduce that \( V(C) \subseteq I(S) \).

    If there are only one source \( u_1 \) and only one sink \( u_2 \) in \( C \), take \( w_1,w_2 \) as above.
    Thus for each \( (u_1,u_2) \)-path there is a \( (w_1,w_2) \)-path containing it.
    If both paths are geodesics we can find a \( (w_1,w_2) \)-geodesic for each one, which gives us \( V(C) \subseteq I(S) \).
    Else we analyse the possibilities for the longest path.
    In any other case let \( v \) be a cut-vertex internal to the longest \( (u_1,u_2) \)-path.
    Without loss of generality let \( (v,v_1) \in A(D) \setminus A(C) \).
    By the Lemmas \ref{lemma.paths} and \ref{lemma.pathsparts} there are \( w_3 \in S \) and a \( (w_1,w_3) \)-geodesic containing the \( (u_1,v) \)-geodesic.
    We then have that all the vertices in this geodesic and the ones in the \( (u_1,u_2) \)-geodesic are in \( I(S) \).
    Since \( u_2,v \in I(S) \) the vertices in the \( (v,u_2) \)-path (which is a geodesic by uniqueness) are in \( I^2(S) \), from where we conclude that \( V(C) \subseteq I^2(S) \).    

    For the last case let \( C \) be a directed cycle which is neither trap nor leaf.
    Thus there are distinct cut vertices \( v_1,v_2 \in V(C) \) and \( (v_1',v_1),(v_2,v_2') \in A(D) \setminus A(C) \).
    By the Lemmas \ref{lemma.paths} and \ref{lemma.pathsparts} there are \( w_1,w_2 \in S \) such that every \( (w_1,w_2) \)-path contains the \( (v_1,v_2) \)-path \( P \), which implies that \( V(P) \subseteq I(S) \).
    Therefore, since \( v_1,v_2 \in I(S) \) all the vertices in the \( (v_2,v_1) \)-path are in \( I^2(S) \).
    
\end{proof}

The proof argues which vertex must be chosen in each unsatisfactory cycle.
All such vertices can be found in linear time.
Thus, $\ohn(D)$ also can be found in linear time, for every oriented cactus $D$.

Theorem~\ref{teo.hcactus} motivated us also to work on the geodetic number for these oriented graphs.
Since a geodetic set is also a hull set, every geodetic set must have a non-extreme vertex of each unsatisfactory cycle.
Besides, as only some satisfactory cycles may have vertices that are not obtained in the first iteration of the interval function, we studied such cycles in order to obtain a minimum geodetic set.
As a result, we define the \textit{falsely satisfactory} cycles, FSC for short.
These can be of two types:
\begin{enumerate}[\textrm{Type} 1:]
    \item \( \ext (C) = \{ u_1,u_2 \} \), \( u_1 \) is a source and \( u_2 \) is a sink.
    The \( (u_1,u_2) \)-paths have distinct lengths, \( P \) being the longest.
    \( P \) has length at least three and one of its internal vertices is a cut-vertex in \( D \).
    Besides, all the following statements hold:
    \begin{enumerate}[(1)]
        \item If there is an RCV \( v_1 \) internal to \( P \), the \( (u_1,v_1) \)-path must have length at least two;
        \item If there is a TCV \( v_2 \) internal to \( P \), the \( (v_2,u_2) \)-path must have length at least two;
        \item If there are both an RCV \( v_1 \) and a TCV \( v_2 \) internal to \( P \), we must have \( P = (u_1, \ldots ,v_2, \ldots ,v_1, \ldots ,u_2) \).
        Moreover, the \( (v_2,v_1) \)-path must also have length at least two.
    \end{enumerate}
    \item The cycle \( C \) is directed and there are distinct RCV \( v_1' \) and TCV \( v_2' \) in \( C \) such that:
    \begin{enumerate}[(1)]
        \item \( d_C(v_2',v_1') \geq 2 \); and
        \item all the other cut-vertices are internal to \( P = (w_0,w_1, \ldots ,w_{k-1},w_k) \), where \( w_0 = v_1' \) and \( w_k = v_2' \).
        Besides, if \( w_i \) is an RCV and \( w_j \) is a TCV then \( i \leq j \) for every \( i,j \in \{ 0,1, \ldots ,k \} \).
    \end{enumerate}
\end{enumerate}
Otherwise, we say that the cycle is \textit{truly satisfactory} and use TSC to simplify.

\begin{lem}\label{lemma.falselysat}
    Let \( D \) be an oriented cactus graph, \( C \subseteq D \) a satisfactory cycle and \( u_1,u_2,v_1,v_2,v_1',v_2' \) as in the above definition.
    Then:
    \begin{itemize}
        \item if \( C \) is truly satisfactory and \( S \) is a minimum hull set of \( D \) then \( I(S) \supseteq V(C) \);
        \item if \( C \) is falsely satisfactory and \( \mathcal{S} = N^+(V(C)) \cup N^-(V(C)) \) then \( I(\mathcal{S}) \nsupseteq V(C) \).
        Moreover, the vertices not in \( I [\mathcal{S}] \) are the following ones:
        \begin{itemize}
            \renewcommand{\labelitemii}{$\star$}
            \item if \( C \) is of type 1, the internal vertices of the \( (w_1,w_2) \)-path where \( w_1 \in \{ u_1,v_2 \} \) and \( w_2 \in \{ u_2,v_1 \} \) are as close as possible;
            \item if \( C \) is of type 2, the internal vertices of the \( (v_2',v_1') \)-path.
        \end{itemize}
        And also none of these vertices is a cut-vertex.
    \end{itemize}
\end{lem}

\begin{proof}
    First we analyse the truly satisfactory cycles.
    We divide our analysis in cases according to Lemma \ref{lemma.extcycles}.
    Let \( C \) have three vertices \( u,v,w \) such that \( u \) is a source, \( v \) is transitive and \( w \) is a sink.
    By Lemmas \ref{lemma.paths} and \ref{lemma.pathsparts} we conclude that \( u,w \in I(S) \).
    If \( v \) is also transitive in \( D \), then it is already in \( S \).
    Else, it is a cut-vertex and using the same lemmas we can see that \( v \in I(S) \).

    If \( C \) is a cycle with at least four extreme vertices, take \( v \in V(C) \) arbitrarily.
    Let \( P \) be a maximal path containing \( v \); if \( P \) is an \( u_1,u_2 \)-path, then it follows that \( u_1 \) is a source and that \( u_2 \) is a sink.
    Again by Lemmas \ref{lemma.paths} and \ref{lemma.pathsparts} we deduce that the vertices of \( P \) are all in \( I(S) \), in particular \( v \).
    Since this argument holds for every vertex of \( C \), we have \( V(C) \subseteq I(S) \).

    Now suppose that \( \ext (C) = \{ u_1,u_2 \} \), with \( u_1 \) source and \( u_2 \) sink.
    By Lemmas \ref{lemma.paths} and \ref{lemma.pathsparts} there are \( w_1,w_2 \in S \) such that every \( (w_1,w_2) \)-path contains one of the two \( (u_1,u_2) \)-paths in \( C \).
    We know that the vertices of an \( (u_1,u_2) \)-geodesic are in \( I(S) \).
    If the two paths are geodesics we are satisfied, else let \( v_1 \) be a hypothetical RCV and \( v_2 \) be a hypothetical TCV, both internal to the longest \( (u_1,u_2) \)-path \( P \).
    Suppose that there is no TCV internal to \( P \), then we can take \( v_1 \) such that \( (u_1,v_1) \in A(C) \).
    By the Lemmas \ref{lemma.paths} and \ref{lemma.pathsparts} there is a \( w_1' \in S \) and a \( (w_1',w_2) \)-geodesic containing the \( (v_1,u_2) \)-path.
    Since this path uses every internal vertex of \( P \) we have \( V(C) \subseteq I(S) \).
    Analogously, if there is no RCV interior to \( P \) we also have \( V(C) \subseteq I(S) \).
    If there are both RCV and TCV internal to \( P \) we can take \( v_1 \) and \( v_2 \) such that all internal vertices of \( P \) are either in the \( (u_1,v_2) \)-path or in the \( (v_1,u_2) \)-path.
    Thus we still have \( V(C) \subseteq I(S) \).

    Now let \( C \) be a directed cycle, neither trap nor leaf.
    Since it is truly satisfactory there must be \( v_1 \neq v_2 \in V(C) \) such that \( v_1 \) is an RCV and \( v_2 \) is a TCV.
    If there are such \( v_1 \) and \( v_2 \) with \( (v_2,v_1) \in A(C) \), we can take \( w_1,w_2 \in S \) such that there is a \( (w_1,w_2) \)-geodesic containing the \( (v_1,v_2) \)-path in \( C \), which in turn contains all the vertices of \( C \).
    If not then take \( v_1,v_2 \)  such that every other cut-vertex of \( C \) lies in the \( (v_1,v_2) \)-path.
    Again by the fact that \( C \) is truly satisfactory, we must have an RCV \( v_3 \) and a TCV \( v_4 \), different from one another, such that the \( (v_1,v_2) \)-path contains the \( (v_4,v_3) \)-path in \( C \).
    Once more by Lemmas \ref{lemma.paths} and \ref{lemma.pathsparts} there are \( w_3,w_4 \in S \) and a \( (w_3,w_4) \)-geodesic containing the \( (v_3,v_4) \)-path in \( C \).
    Knowing that the \( (v_4,v_3) \)-path is contained in the \( (w_1,w_2) \)-path, we thus conclude that all the vertices of \( C \) either in the \( (v_3,v_4) \)-path or in the \( (v_4,v_3) \)-path will be in \( I(S) \), which implies \( V(C) \subseteq I(S) \).

    We next treat the cases where \( C \) is an FSC.
    Take \( C \) of type 1 and let \( u_1 \) and \( u_2 \) respectively be the source and the sink in \( C \).
    Due to previous arguments, it is straightforward that the vertices internal to the \( (u_1,u_2) \)-geodesic are in \( I(\mathcal{S}) \).
    If there are only RCV's in the other \( (u_1,u_2) \)-path, with \( v_1 \) being the one such that \( d_C(u_1,v_1) \) is as small as possible, there is at least one vertex in the \( (u_1,v_1) \)-path which will not be in \( I(\mathcal{S}) \).
    The case in which there are only RCV's is analogous.
    We thus suppose that there are both RCV's and TCV's and let \( v_1 \) be an RCV and \( v_2 \) be a TCV such that \( d_C(v_2,v_1) \) is the smallest value possible.
    By repetitive arguments we state that the vertices in the \( (u_1,v_2) \)-path and in the \( (v_1,u_2) \)-path are in \( I(\mathcal{S}) \).
    However, knowing that there is at least one vertex in the \( (v_2,v_1) \)-path we have that this will not be in \( I(\mathcal{S}) \).

    Assume now \( C \) of type 2, thus there are \( v_1 \) RCV and \( v_2 \) TCV such that the other cut vertices are in the \( (v_1,v_2) \)-path \( P \).
    Moreover if \( P = (v_1=u_0,u_1, \ldots ,u_k = v_2) \) and \( u_i \) and \( u_j \) are respectively an RCV and a TCV we have \( i \leq j \).
    Using again Lemma \ref{lemma.paths} and Lemma \ref{lemma.pathsparts} we have that all the vertices in any \( (u_i,u_j) \)-path are in \( I(\mathcal{S}) \), and these are the only ones.
    Since the \( (v_1,v_2) \)-path contains every one of these vertices and there is at least one vertex internal to the \( (v_2,v_1) \)-path in \( C \) we have \( V(C) \nsubseteq I(\mathcal{S}) \).

    To close this demonstration let \( v \) be a vertex in an FSC not in \( I(\mathcal{S}) \).
    If it had any neighbor outside \( C \) he would be a cut-vertex, thus contradicting the definition of FSC.
\end{proof}

Since, for each FSC, the vertices not in \( I(S) \) are not connected to the rest of the graph, we conclude that every geodetic set must have at least one non-extreme vertex of each FSC.

\begin{thm}\label{theo.gcactus}
    Let \( D \) be an oriented cactus graph.
    There is a geodetic set composed by all the extreme vertices and one non-extreme of each unsatisfactory and falsely satisfactory cycle.
    Moreover, this geodetic set is minimum.
\end{thm}

\begin{proof}
    By Theorem \ref{teo.hcactus}, we can build a minimum hull set \( S \) only with the extreme vertices and one non-extreme form each unsatisfactory cycle.
    Besides, the only vertices not in \( I(S) \) are in satisfactory cycles.
    Moreover, by Lemma \ref{lemma.falselysat} we can be even more specific and say which vertices are these.
    So we only have left to show that adding the remaining vertices in the statement to \( S \) will give us a minimum geodetic set.
    
    If \( C \) is a falsely satisfactory cycle of type 1, let \( u_1,u_2,v_1,v_2 \) be as in the definition.
    If there is not any TCV in the longest \( (u_1,u_2) \)-path of \( C \), take \( w \in N^+(u_1) \) and add it to \( S \).
    Once more by the Lemmas \ref{lemma.paths} and \ref{lemma.pathsparts} we can take \( w_2 \in S \) such that there is a \( (w,w_2) \)-geodesic containing the \( (w,v_2) \)-path, which gives us the desired result.
    The case where there are no RCV's is analogous.
    If we have both RCV and TCV, take \( v_1 \) and \( v_2 \) as close as possible and take \( w \) an internal vertex of the \( (v_2,v_1) \)-path, which exists by definition.
    We can take \( w_1,w_2 \in S \) such that there are a \( (w_1,w) \)-geodesic and a \( (w,w_2) \)-geodesic respectively containing the \( (v_1,w) \)-path and the \( (w,v_2) \)-path in \( C \).
    
    Now let \( C \) be of type 2 and \( v_1',v_2' \) be as in the definition.
    Take \( v \in N^+(v_2') \cap V(C) \).
    Notice that there are \( w \in S \) and a \( (v,w) \)-geodesic containing the \( (v,v_2') \)-path in \( C \).
    Since the last path contains all vertices of \( C \) we have \( V(C) \subseteq I(S) \).
\end{proof}

Once more, such minimum geodetic set can be found in linear time by just analyzing the cycles and determining which ones are (truly/falsely) satisfactory and unsatisfactory.

\section{Further research}
\label{section.conc}

We first proved that the hull number of an arbitrary oriented graph $D$ is at most the number of extreme vertices of $D$ plus $\frac{2}{3}$ of the non-extreme.
A natural question is whether a similar bound holds for the geodetic number. 

Here we also proved that, given an oriented graph $D$, determining $\ohn(D)$ and $\ogn(D)$ are \NP-hard problems even if the underlying graph of $D$ is bipartite.
Equivalent results were known in the literature~\cite{CIPCTHN,SROTGNOAG} for the undirected case.
We also proved that computing the geodetic number of an oriented split graph is \NP-hard, while the undirected version can be solved in linear time~\cite{SROTGNOAG}.
We finally proved that computing the geodetic number of an oriented cobipartite graph is also \NP-hard.

It is known that determining $\hn(G)$ is a \NP-hard problem even if $G$ is chordal~\cite{TGHNIHFCG}.
Another open problem would be:

\begin{conj}
    Given an oriented graph $D$ and a positive integer $k$, then determining whether $\ohn(D) \leq k$ is \NP-complete, even if the underlying graph of $D$ is chordal.
\end{conj}
In fact, even determining such parameters for tournaments seems a hard task.
Another natural problem is to find some graph class $\mathcal{G}$ for which determining $\ohn(D)$ is a \NP-hard problem, while determining $\hn(G)$ can be solved in polynomial time, for some simple graph $G \in \mathcal{G}$ and some orientation $D$ of $G$.

Finally, bounds and complexity results for other graph classes are also widely open. One could also study such parameters in other graph convexities. We should emphasize that our bound in Section~\ref{section.st} and the reductions to bipartite and cobipartite graphs also hold for the corresponding parameters in the two-path convexity~\cite{PWW2008}.

\bibliography{oriented-hull-geodetic}

\end{document}